\documentclass{gtpart}

\title[Tied monoids]{Tied monoids}

\author[D. Arcis]{Diego Arcis}
\address{Facultad de Ciencias de la Salud, Universidad Aut\'onoma de Chile - Sede Talca, 5 Poniente 1670, Talca 3460000, Chile.}
\email{diego.arcis@uautonoma.cl}

\author[J. Juyumaya]{Jes\'us Juyumaya}
\address{IMUV, Universidad de Valpara\'iso, Gran Breta\~{n}a 1111, Playa Ancha, Valpara\'iso 2340000, Chile.}
\email{juyumaya@gmail.com}


\usepackage[mathscr]{euscript}
\usepackage[dvipsnames]{xcolor}
\usepackage{amsmath,amssymb,amsfonts} 
\usepackage{pinlabel}
\usepackage{graphicx}

\usepackage{url}

\usepackage{cancel,soul}

\usepackage{ytableau}

\newtheorem{thm}{\bf Theorem}[section]
\newtheorem{crl}[thm]{\bf Corollary}
\newtheorem{dfn}[thm]{\bf Definition}
\newtheorem{lem}[thm]{\bf Lemma}
\newtheorem{pro}[thm]{\bf Proposition}

\theoremstyle{definition}
\newtheorem{rem}{\bf Remark}[section]

\newtheorem*{exm}{\bf Example}
\newtheorem*{ntn}{\bf Notation}

\numberwithin{equation}{section}

\makeatletter
\renewcommand{\thefigure}{\ifnum \c@section>\z@ \thesection.\fi
\@arabic\c@figure}
\@addtoreset{figure}{section}
\makeatother

\begin{document}

\def\N{\mathbb{N}}

\def\b{{\rm b}}

\def\uu{\mathsf{u}}
\def\vv{\mathsf{v}}
\def\xx{\mathsf{x}}
\def\yy{\mathsf{y}}
\def\pp{\mathsf{p}}
\def\qq{\mathsf{q}}
\def\rr{\mathsf{r}}
\def\ss{\mathsf{s}}

\def\lg{{\rm{lg}}}
\def\sb{{\rm{sb}}}

\def\rev{\mathrm{rev}}

\def\A{{\rm{A}}}
\def\B{{\rm{B}}}
\def\D{{\rm{D}}}
\def\X{{\rm{X}}}

\def\BB{{\mathcal B}}
\def\EE{{\mathcal E}}
\def\PP{{\mathcal P}}

\def\SS{{\mathfrak S}}

\def\End{{\mathrm{End}}}
\def\Fix{{\mathrm{Fix}}}

\def\blue{\color{blue}}
\def\green{\color{green}}
\def\red{\color{red}}

\newcommand{\tf}[1]{$\mathfrak{T}_{#1}$} 

\def\th{\theta}
\def\oth{\bar{\theta}}
\def\tb{\Theta}
\def\td{\bar{\Theta}}
\def\oa{\bar{a}}
\def\ob{\bar{b}}
\def\od{\bar{d}}

\def\e{\epsilon}
\def\ve{\varepsilon}

\newcommand{\stblue}[1]{\textcolor{blue}{\st{#1}}}
\newcommand{\stred}[1]{\textcolor{red}{\st{#1}}}

\newcommand{\cancelblue}[1]{\textcolor{blue}{\cancel{#1}}}
\newcommand{\cancelred}[1]{\textcolor{red}{\cancel{#1}}}

\newcommand\figureone{
	\includegraphics{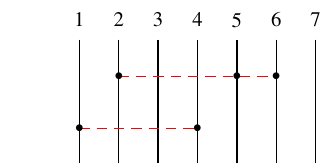} 
}

\newcommand\figuretwo{
	\includegraphics{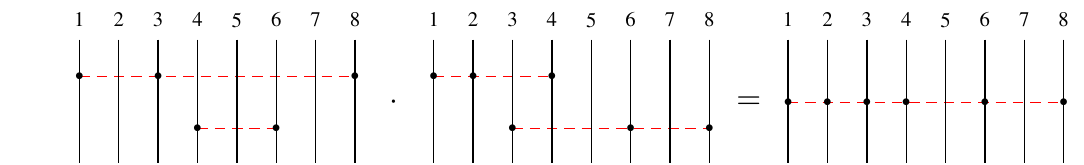} 
}

\newcommand\figurefou{
	\includegraphics{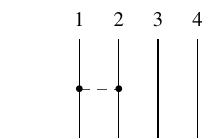}\includegraphics{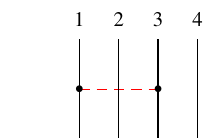}\includegraphics{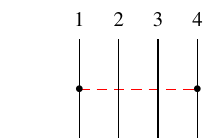}\includegraphics{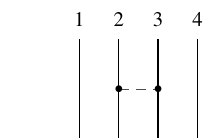}\includegraphics{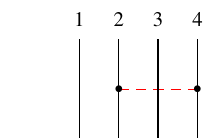}\includegraphics{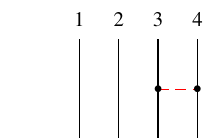} 
}

\newcommand\figurefiv{
}

\newcommand\figuresix{
	\includegraphics{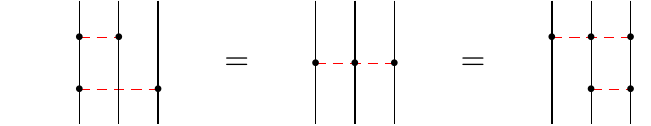} 
}

\newcommand\figuresev{
	\begin{array}{c}\includegraphics{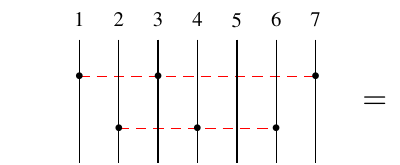}\\[1.5mm]\includegraphics{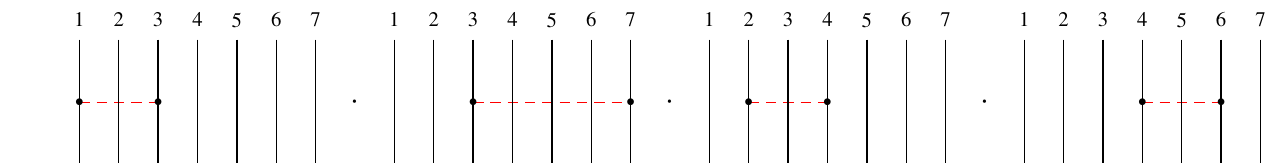}\end{array} 
}

\newcommand\figureeig{
	\includegraphics{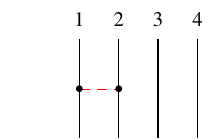}\includegraphics{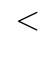}\!\!\!\!\!\!\!\!\!\!\includegraphics{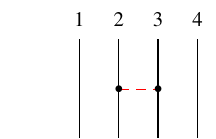}\includegraphics{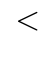}\!\!\!\!\!\!\!\!\!\!\includegraphics{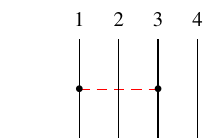}\includegraphics{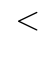}\!\!\!\!\!\!\!\!\!\!\includegraphics{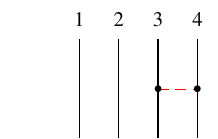}\includegraphics{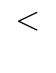}\!\!\!\!\!\!\!\!\!\!\includegraphics{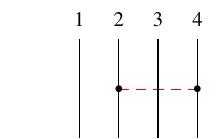}\includegraphics{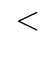}\!\!\!\!\!\!\!\!\!\!\includegraphics{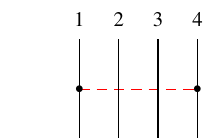} 
}

\newcommand\figurenin{
	\includegraphics{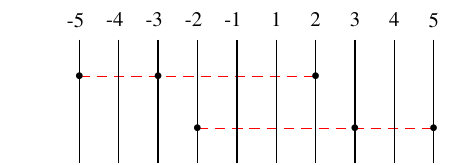} 
}

\newcommand\figureten{
	\begin{array}{ccc}\includegraphics{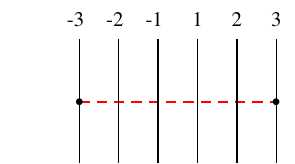}&\includegraphics{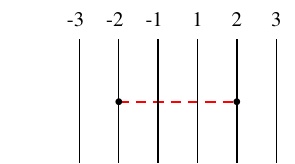}&\includegraphics{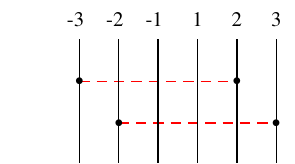}\\[1.5mm]\includegraphics{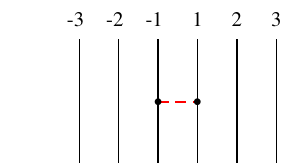}&\includegraphics{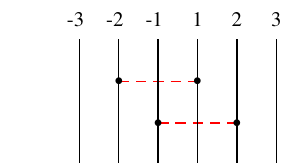}&\includegraphics{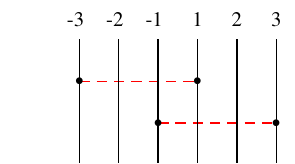}\\[1.5mm]\includegraphics{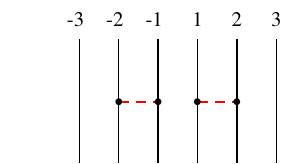}&\includegraphics{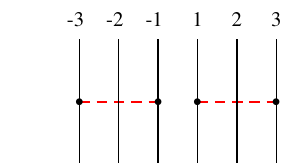}&\includegraphics{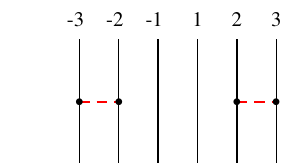}\end{array} 
}

\newcommand\figureele{
	\includegraphics{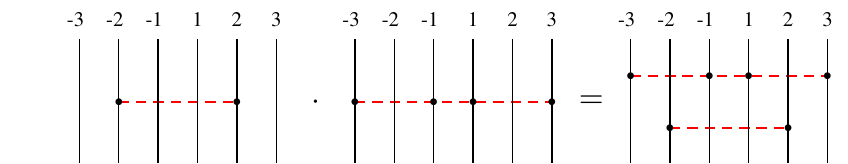} 
}

\newcommand\figuretwe{
	\includegraphics{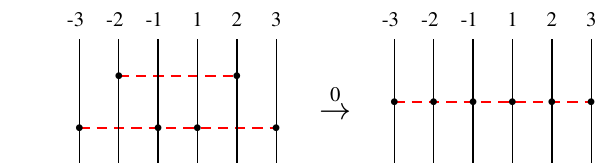} 
}

\newcommand\figuretenthr{
	\includegraphics{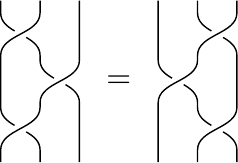}\qquad\qquad\includegraphics{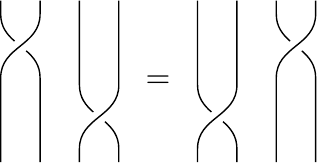} 
}

\newcommand\figuretenfou{
	\begin{array}{c}\includegraphics{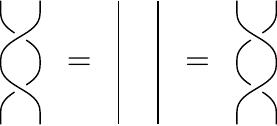}\quad\includegraphics{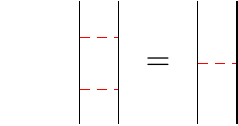}\quad\includegraphics{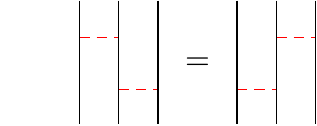}\\[2mm]\includegraphics{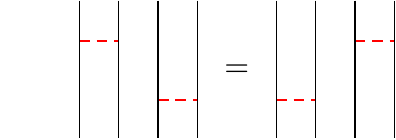}\qquad\includegraphics{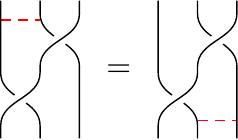}\qquad\includegraphics{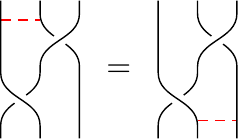}\qquad\\[2mm]\includegraphics{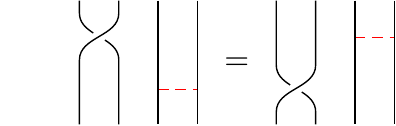}\qquad\includegraphics{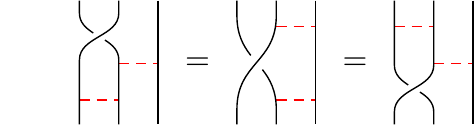}\qquad\end{array} 
}

\newcommand\figuretenfiv{
	\includegraphics{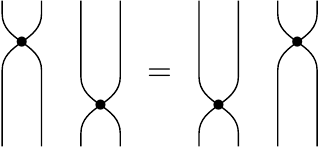}\qquad\includegraphics{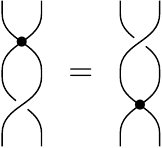}\qquad\includegraphics{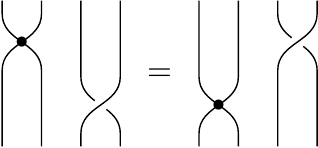}\qquad\includegraphics{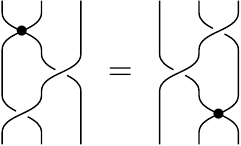} 
}

\newcommand\figuretensix{
	\begin{array}{c}\includegraphics{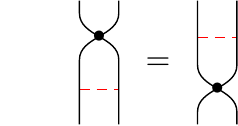}\quad\includegraphics{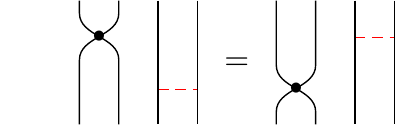}\quad\includegraphics{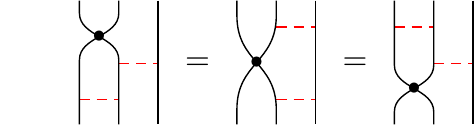}\\[2mm]\includegraphics{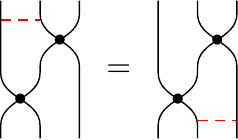}\qquad\includegraphics{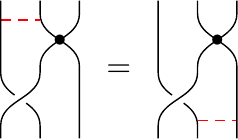}\qquad\includegraphics{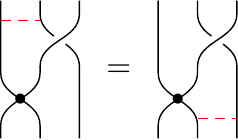}\quad\includegraphics{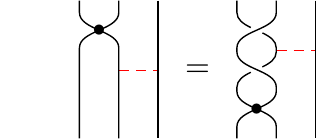}\end{array} 
}

\newcommand\figuretensev{
	\begin{array}{c}\includegraphics{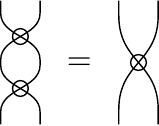}\qquad\includegraphics{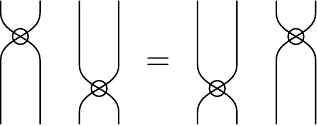}\qquad\includegraphics{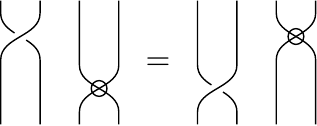}\\[0.5cm]\includegraphics{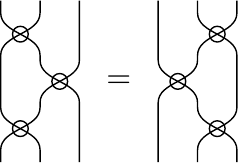}\qquad\includegraphics{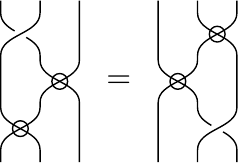}\end{array} 
}

\newcommand\figureteneig{
	\begin{array}{c}\includegraphics{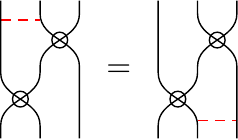}\qquad\includegraphics{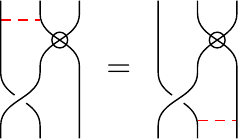}\qquad\includegraphics{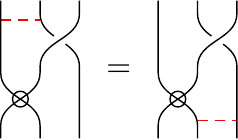}\quad\includegraphics{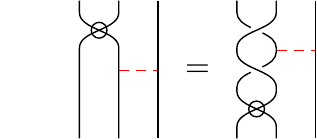}\\[0.5cm]\includegraphics{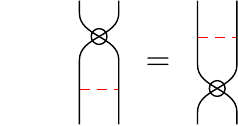}\quad\includegraphics{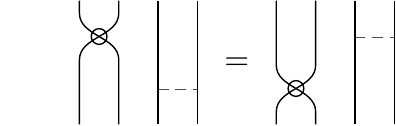}\quad\includegraphics{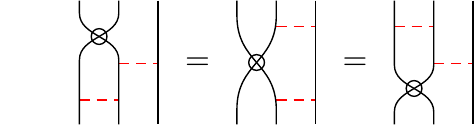}\end{array} 
}

\newcommand\figuretennin{
	\includegraphics{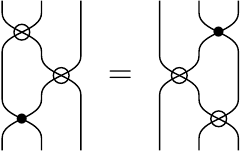}\qquad\quad\includegraphics{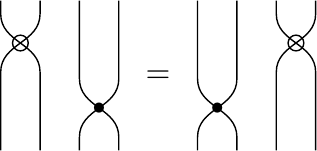} 
}

\begin{abstract}
We construct certain monoids, called \emph{tied monoids}. These monoids result to be semidirect products finitely presented and commonly built from braid groups and their relatives acting on monoids of set partitions. The nature of our monoids indicate that they should give origin to new knot algebras; indeed, our tied monoids include the tied braid monoid and the tied singular braid monoid, which were used, respectively, to construct new polynomial invariants for classical links and singular links. Consequently, we provide a mechanism to attach an algebra to each tied monoid. To build the tied monoids it is necessary to have presentations of set partition monoids of types A, B and D, among others. For type $A$ we use a presentation due to FitzGerald and for the other type it was necessary to built them.
\end{abstract}

\maketitle

\setcounter{tocdepth}{2}
\tableofcontents

\section*{Introduction}

Tied links and tied braids were introduced in \cite{AiJu16} as a generalization of classical links and classical braids. These generalizations were used to define polynomial invariants for classical links that result to be more powerful than the Homflypt and Kauffman polynomials, respectively, see \cite{AiJu16} and \cite{AiJu18}, cf. \cite[Subsection 9.2]{CJKL18}. The algebraic counterpart of tied links is the {\it tied braid monoid} \cite{AiJu16} in the same spirit as the braid group is the algebraic counterpart of classical links. The  tied braid monoid was originally defined through generators and relations whose source is purely diagrammatic; namely, these defining relations come out from the monomial defining relations of the diagram algebra called \emph{algebra of braids and ties} \cite{AiJu00,RyH11,Ma17,JaPo17}, shorted \emph{bt--algebra}, so that,  this algebra becomes a quotient of the monoid algebra of the tied braid monoid \cite[Remark 4.3]{AiJu16}.

Denote by $T\BB_n$ the tied braid monoid with $n$ strands and by $P_n$ the monoid of set partitions of the set $\{1,\ldots,n\}$. It was proved in \cite{PreAiJu18} that $T\BB_n$  can be decomposed as  the semidirect product $P_n\rtimes_{\rho}\BB_n$, between $P_n$ and the braid group $\BB_n$ on $n$ strand, corresponding to the action $\rho$ of $\BB_n$ on $P_n$, derived from the natural epimorphism from $\BB_n$ to the symmetric group on $n$ symbols. Also in \cite{PreAiJu18} it was the tied singular braid monoid $TS\BB_n$ defined and proved that it also has a decomposition as a semidirect product, this time between the monoid $P_n$ and the singular braid monoid $S\BB_n$. 

This paper deals with the construction of tied monoids by applying the Lavers' method \cite{La98}, that is, the construction of finitely presented  semidirect products of monoids from finite presentations of its monoid factors. More precisely, here we construct the tied monoid $T^{\Gamma}\!M$, that is the semidirect product $P\rtimes_{\rho}M$ where $M$ is taken as the Artin group of type $\Gamma\in\{\A,\B,\D\}$, the singular braid monoid \cite{Sm88,Ba92,Bi93} or the virtual braid group \cite{KaLa04,Kd04}, the action $\rho$ is one factorized through an homomorphism of $M$ onto the Coxeter group of type $\Gamma$ and the factor $P$ is an appropriate monoid of set partitions. Observe that the Lavers method applied to the Artin group $\BB_n$ (resp. $S\BB_n$) together with the monoid $P_n$ results to be the monoid $T\BB_n$ (resp. $TS\BB_n$). The monoids $T\BB_n$ and $TS\BB_n$ play the  algebraic role, as $\BB_n$ and $S\BB_n$ for tied link and tied singular links respectively; for details see \cite{AiJu16,PreAiJu18}. 

As was mentioned in the abstract a motivation to introduce the tied monoids is to apply them to knot theory via the construction of knot algebras. Thus we propose also a mechanism that attach an algebra, called \emph{tied algebra}, to each tied monoid, cf. \cite{Ma17,Ma18}. These algebras are obtained as a quotient of the monoid algebra of the tied monoids factoring out certain quadratic relations, cf. \cite[Eq. (20)]{AiJu19}. The mechanism captures the classical Hecke algebras, the bt--algebra defined in \cite{AiJu00}, as well the bt--algebras defined by Flores \cite{Fl19} and Marin \cite{Ma17}.

We are going to describe now as the paper is organized. Section \ref{S1} gives the tools from the monoid theory that will be used during the paper, including a Lavers' theorem on presentations of semi direct product of monoids; this section also gives the definition of the poset monoid of set partitions (or simply partitions) associated to a given set, and the general definition of a tied monoid. Section \ref{S2} starts by studying the monoid of partitions of type A, that is $P_n$. In \cite[Theorem 2]{Fi03}, D.G. FitzGerald has determined a presentation of $P_n$, see here Theorem \ref{001}. This theorem is used in Section \ref{S3} to construct presentations of the set partition monoids of types B and D, see respectively, Theorem \ref{087} and Theorem \ref{106}. To get these presentations it was necessary to consider the monoids of signed partitions and restricted signed partitions; thus presentations of these monoids are also constructed, see respectively Theorem \ref{067} and Theorem \ref{107}. All these theorems on presentations are interesting in their own right and seem to be absent in the literature. We observe that as a consequence of the found presentations we obtain normal forms in all monoids mentioned above.

Section \ref{S4}  deals with the construction, through the Lavers' method, of several tied monoids of type A. Firstly, we will see that both the tied braid monoid and the  tied singular braid monoid, defined in \cite{AiJu19}, can be recovering using this method, see Remark \ref{015} and Remark \ref{016} respectively. Secondly, two new tied monoids are introduced: the \emph{tied virtual braid monoid} (Theorem \ref{028}) and the \emph{tied virtual singular braid monoid} (Theorem \ref{017}).

In Section \ref{S5} we  use the Lavers' method to built  presentations of tied monoids related to the  Artin monoids of type B and D. To be precise,  in Theorem \ref{031}  we give a presentation of the tied monoid of type B, namely $T^B\!\BB_n^B=P_n^B\rtimes_{\rho}\BB_n^B$, where $P_n^B$ is the  monoid of B-partitions defined by Reiner \cite{Re97}, $\BB_n^B$ is the braid group of type B and  $\rho$ is  an action that is factorized with the natural epimorphism from $\BB_n^B$ to the Coxeter group of type B. The tied monoid of type D, namely $T^D\!\BB_n^D=P_n^D\rtimes_{\rho}\BB_n^D$, is  constructed similarly and addressed in Theorem \ref{TMD}.

Section \ref{S6} proposes a mechanism of construction of tied algebras, that is, algebras derived from tied monoids: as an example, we show as works the mechanism to obtain the bt--algebra. Also we show in detail as works the mechanism to obtain the tied algebra attached to the tied monoid $T^B\!\BB_n^B$.

\section{Preliminaries}\label{S1}

This section consists of the three subsections. The first one is on monoid presentations, the second one is on set partitions and the last one gives the definition of a tied monoid. 

\subsection{Monoid presentations}\label{095}

Given a monoid $M$ and $R\subseteq M\times M$, we will denote by $\overline{R}$ the \emph{congruence closure} of $R$, i.e. $\overline{R}$ is the minimal equivalence relation containing $R$ that is invariant respect to the product of $M$. Two elements $g,g'\in M$ will be called $R$\emph{-equivalent} and denoted by $g\equiv g'$, if $(g,g')\in\overline{R}$. Recall that, for a homomorphism of monoids $f:M\to N$, the set $\ker(f):=\{(g,g')\mid f(g)=f(g')\}$ is a congruence on $M$ called the \emph{kernel congruence}. 

\begin{pro}[Lallement {\cite[Isomorphism Theorem, p. 10]{Ll79}}]\label{105}
If $f:M\to N$ is an epimorphism, then $M/\ker(f)\simeq N$.
\end{pro}

For a set $A$ we denote by $A^*$ the free monoid generated by $A$. The length of a word $u\in A^*$ will be denoted by $\lg(u)$.

A pair $\langle A\mid R\rangle$ with $R\subseteq A^*\times A^*$ is said to be a \emph{presentation} for a monoid $M$, if $M$ is isomorphic to $A^*/\overline{R}$. Elements of $A$ are called generators and those of $R$ defining relations. We say that $M$ is \emph{finitely presented} if it has a presentation $\langle A\mid R\rangle$ with $A$ and $R$ finite.

\begin{pro}[Ru\v{s}kuc {\cite[Proposition 2.3]{Ru95}}]\label{009}
Let $M$ be a monoid generated by $A$, and let $R\subseteq A^*\times A^*$. Then $M=\langle A\mid R\rangle$ if and only if the following conditions hold:
\begin{enumerate}
\item We have $\overline{u}=\overline{u}'$ for all $(u,u')\in R$,
\item If $\overline{u}=\overline{u}'$ for some $u,u'\in A^*$, then $u,u'$ are $R$-equivalent.
\end{enumerate}
where $\overline{u},\overline{u}'$ denote the words $u,u'$ considered as elements of $M$.
\end{pro} 
We say that $r\in A^*\times A^*$ is a \emph{consequence} of $R$ if $r\in\overline{R}$. A \emph{transformation of type} \tf{1} is a move that transforms $\langle A\mid R\rangle$ into $\langle A\mid R\cup\{r\}\rangle$ for some $r$ that is consequence of $R$. Inversely, a \emph{transformation of type} \tf{2} transforms $\langle A\mid R\rangle$ into $\langle A\mid R\backslash\{r\}\rangle$ for some $r$ that is consequence of $R\backslash\{r\}$. A symbol $a\in A$ is said to be a \emph{consequence} of $R$ if $(a,u)\in\overline{R}$ for some word $u\in A^*$. A \emph{transformation of type} \tf{3} is a move that transforms $\langle A\mid R\rangle$ into $\langle A\cup\{a\}\mid R\cup\{(a,u)\}\rangle$ for some $a\not\in A^*$ and $u\in A^*$. Note that $a$ is a consequence of $R\cup\{(a,u)\}$. A \emph{transformation of type} \tf{4} transforms $\langle A\mid R\rangle$ into $\langle A\backslash\{a\}\mid S\rangle$, where $S$ is obtained from $R$ by removing a pair $(a,u)$ or $(u,a)$ for some $u\in A^*$, and replacing every other appearance of $a$ in $R$ by $u$, that is, every $(a,u')$ or $(u',a)$ in $R$ for some $u'\in A^*\backslash\{u\}$ is replaced by $(u,u')$ or $(u',u)$ respectively. Transformations of type \tf{1} to \tf{4} are called \emph{Tietze transformations}.

\begin{pro}[Ru\v{s}kuc {\cite[Proposition 2.5]{Ru95}}]\label{012}
Two finite presentations define the same monoid if and only if one can be obtained from the other by a finite number of applications of Tietze transformations.
\end{pro}

Let $M,N$ be two monoids, and let $\rho:M\to\End(N)$ be a monoid action, that is a unitary monoid homomorphism. The semi direct product of $M$ with $N$ respect to this action, denoted by $N\rtimes_{\rho} M$, is the monoid $N\times M$ with the following product:\[(b,a)(d,c)=(b\rho_a(d),ac)\qquad a,c\in M,\,b,d\in N.\]

\begin{thm}[Lavers \cite{La98}]\label{013}
Let $\rho:M\to\End(N)$ be a monoid action, where $M$ and $N$ are monoids finitely presented, respectively, by $\langle A\mid R\rangle$ and $\langle B\mid S\rangle$. Then 
\[N\rtimes_{\rho} M\simeq\left\langle A\cup B\mid R\cup S\cup\{(ba,a\rho_a(b))\mid a\in A\text{ and }b\in B\}\right\rangle.\]
\end{thm}

The rest of the material of this subsection belong to commutaive monoids theory and is taken from \cite{RoGa99}.

An \emph{admissible order} on a commutative monoid is an invariant total order in which the identity is the smallest element.

In what follows $X:=\{x_1<\ldots<x_n\}$ denotes a finite totally ordered set, and $CX^*$ the free commutative monoid generated by it. 

For $g,h\in CX^*$ we say that $h$ \emph{divides} $g$, denoted by $h|g$, if $g=g'h$ for some $g'\in CX^*$. The element $g'$ is denoted by $gh^{-1}$. Since $CX^*$ is commutative, for every $g\in CX^*$ there are unique $e_1,\ldots,e_n\in\N$ such that $g=x_1^{e_1}\cdots x_n^{e_n}$. We call $(e_1,\ldots,e_n)$ the \emph{exponent} of $g$ and denote it by $\exp(g)$. Note that $h|g$ if and only if $\exp(g)-\exp(h)\in\N^n$. The right lexicographic order on $\N^n$ induces a natural admissible order $<$ on $CX^*$: $g<h$ if $\exp(g)$ is smaller than $\exp(h)$, that is, $h|g$.

Set $R:=\{(a_1,b_1),\ldots,(a_r,b_r)\}$ a subset of $CX^*\times CX^*$. We say that $R$ is \emph{reduced} if the following conditions hold:
\begin{enumerate}
\item $b_i<a_i$ for all $i\in[r]$.
\item $a_i$ does not divide $a_j$ for all $i,j\in[r]$ with $i\neq j$.
\item $b_i$ does not divide $a_j$ for all $i,j\in[r]$.
\end{enumerate}

In the sequel we assume that $R$ is reduced. 

The \emph{normal form} of $R$ is the map $N:CX^*\to CX^*$ defined as follows:\[N(g)=\left\{\begin{array}{ll}
g&\text{if }a_i\text{ does not divide }g\text{ for all }i\in[r],\\
N(ga_m^{-1}b_m)&\text{if }m\text{ is minimal in }[r]\text{ satisfyng }a_m|g.
\end{array}\right.\]The element $N(g)$ is called the $R$--\emph{normal form} of $g$.

\begin{dfn}[Cf. Rosales--Garc\'ia {\cite[Theorem 6.7]{RoGa99}}]\label{081}
We call $R$ a \emph{canonical system of generators} of $\overline{R}$ if $N(g)=N(h)$ 
for all $g,h\in CX^*$ satisfying $(g,h)\in\overline{R}$. 
\end{dfn}

\begin{thm}[Rosales--Garc\'ia {\cite[Section 14.3]{RoGa99}}]\label{093}
Let $Y$ be a subset of $X$, and assume that $R$ is a canonical system of generators. Then, the submonoid of $\langle X\mid R\rangle$ generated by $Y$ is presented by $\langle Y\mid S\rangle$ where $S=R\cap(CY^*\times CY^*)$.
\end{thm}

\subsection{Set partitions}A \emph{set partition} (or simply \emph{partition}) of a nonempty set $A$ is a collection $I$ of subsets $I_1,\ldots,I_k$ of $A$, called \emph{blocks} of $I$, such that $I_1\cup\cdots\cup I_k=A$ and $I_i\cap I_j=\emptyset$ for all different $i,j\in[k]$. Blocks with a unique element are called \emph{singleton blocks}. There is a partial order $\preccurlyeq$ on the collection of all set partitions of $A$ such that $I\preccurlyeq J$ if each block of $J$ is a union of blocks of $I$. The collection set formed by the set partitions of $A$, which is denoted by $P(A)$, has a structure of idempotent commutative monoid with $1=\{\{a\}\mid a\in A\}$ and where the product $IJ$ of the set partitions $I$ with $J$ is defined as\[IJ=\min\{K\in P(A)\mid I\preccurlyeq K\text{ and }J\preccurlyeq K\}.\]Note that for all set partition $I$ of $A$, we have$1\preccurlyeq I\preccurlyeq\{A\}$.

In what follows $(A,<)$ denotes a finite totally ordered set. This relation yields a total order $<$ on the power set of $A$. Indeed, if $X,Y$ are subsets of $A$, then $X<Y$ if and only if either $|X|<|Y|$ or $|X|=|Y|$ with $\min(X\backslash(X\cap Y))<\min(Y\backslash(X\cap Y))$. We shall write a set partition $I$ of $A$ by $I=(I_1,\ldots,I_k)$ when its blocks $I_1,\ldots,I_k$ are listed in increasing order according to this total order, that is $I_1<\cdots<I_k$.

Write $A=\{a_1<\cdots<a_n\}$. We shall say that $I\in P(A)$ {\it is symmetric} if $a_i$ and $a_j$ belong to the same block, then $a_{n+1-i}$ and $a_{n+1-j}$ are in the same block as well, see \cite{Or00}. Note that the definition of a symmetric set partition depends on the choice of the total order on $A$. We will denote by $\SS P(A)$ the subset of all symmetric set partitions in $P(A)$.

To study the product of $P(A)$ we need to introduce the following definition.

\begin{dfn}
Set $I=\{I_1<\cdots<I_k\}$ to be a collection of subsets of $A$. For every $m\in \mathbb{Z}_{\geq 0}$ we define inductively $(I)^m$ as follows: $(I)^0=I$,\[(I)^1=\left\{\begin{array}{ll}
I&\text{if }I_i\cap I_j=\emptyset\text{ for all }i,j\in [k]\text{ with }i\neq j,\\
\{I_p\cup I_q\}\cup I\backslash\{I_p,I_q\}&\text{if }I_p,I_q\text{ are } < \text{-minimal in }I\text{ such that }I_p\cap I_q\neq\emptyset;\end{array}\right.\]and $(I)^m=((I)^{m-1})^1$ for $m>1$.
\end{dfn}

Observe that for every collection of subsets $I$, there exists $m \geq1$ such that $(I)^m$ is a set partition of some subset of $A$. We define the {\it exponent} of $I$, denoted by $e(I)$, as the minimal $m$ for which $(I)^m$ is a set partition. Notice that $e(I)=1$ if $I$ is a set partition.

\begin{pro}\label{000}
We have $IJ= (I\cup J)^{e(IUJ)}$ for all set partitions $I,J$ of $A$.
\end{pro}
\begin{proof}
Let $T=(I\cup J)^{e(I\cup J)}$. By definition $I\preccurlyeq T$ and $J\preccurlyeq T$, so $IJ\preccurlyeq T$. Let $E$ be a block of $T$. For every $a\in E$, $E$ contains the unique blocks $I_a$ of $I$ and $J_a$ of $J$ that contain $a$. This implies that every block of $T$ is contained in a block of $IJ$, hence $T\preccurlyeq IJ$.  Therefore $IJ=T$. 
\end{proof}

\subsection{Tied monoids}

Below the definition of a tied monoid which is one of the main objects introduced in the paper.

\begin{dfn}\label{155}
Let $M$ be a monoid, and let $P$ be a commutative idempotent monoid endowed with an action $\rho$ of $M$ on $P$ which is factored by a monoid homomorphism from $M$ onto a Coxeter group of type $\Gamma$. The \emph{tied monoid} of $M$ respect to $P$, denoted by $T^{\Gamma}\!M$, is defined as the semidirect product $P\rtimes_{\rho} M$.
\end{dfn}
We will omit $\Gamma$ in $T^{\Gamma}\!M$ whenever the Coxeter group is of type $\Gamma=\A$.

\begin{exm}
Consider $M=\mathbb{Z}$ and $P=\langle y\mid y^2=y\rangle$ with the action $\rho:M\to\End(P)$ that factorizes trivially onto $W=\mathbb{Z}/2\mathbb{Z}$. The tied monoid $TM=\mathbb{Z}\rtimes_{\rho}P$ is presented with generators $x,y,x^{-1}$ subject to the relations $xy=yx$ and $y^2=y$.
\end{exm}

In this paper $P$ will be taken as a monoid of set partitions or its relatives.

\begin{rem}[Involutions]
Under the hypothesis of Definition \ref{155}, denote by $\pi:M\to W$ the monoid homomorphism which $\rho$ is factored. Let $s\in W$ such that $s^2=1$. Then, for $(a,x)\in P\times M$ such that $\pi(x)=s$, we have $ax^2=x^2\rho_{x^2}(a)=x^2a$ in $T^{\Gamma}M$. Note that it occurs in particular for Coxeter generators of $W$.
\end{rem}

\section{Monoids of set partitions of type A}\label{S2}

In this section we recall a theorem due to FiztGerald which gives a presentation of the monoid $P(A)$ with $A$ finite, see Theorem \ref{001}.  This theorem is used several times during the  paper, for instance,  it is used to  determine   a presentation for the so-called monoid of signed  partitions (Theorem \ref{067}). In turn,  we use Fitzgerald's theorem to prove one of our main results: a presentation of the  monoid of set partitions of type B. Also, here is introduced the representation by ties for the set partitions, which, in this context, turns out to be more convenient than the other known representations. In fact, the topological flavour of the ties is justified in Section \ref{S4}.

\subsection{The monoid $P_n$}
 
Let $n$ be an integer greater than one, and let $[n]=\{1,\ldots,n\}$. A \emph{partition} of type $A_n$ or simply an $A_n$-\emph{partition} is a set partition of $[n]$. Along of the present paper $P([n])$ will be simply denoted by $P_n$.

Consider now the usual total order $<$ on $[n]$. As was explicated before, this order yields a total order on the power set of $[n]$. So, every collection of subsets of $[n]$ can be ordered according to $<$.

Sometimes it is very useful to represent $A_n$-partitions using linear graphs or ties between strings. A \emph{linear graph} of an $A_n$-partition is the graph whose vertices are the points of $[n]$ and its set of edges are arcs that connect the point $i$ with the point $j$, if $j$ is the minimum in the same block of $i$ satisfying $i<j$. For example, the graph below represents the $A_7$-partition $I=\{\{1,4\},\{2,5,6\},\{3\},\{7\}\}$.\begin{center}
\scalebox{.65}{\begin{picture}(120,25)
\put(0,2){$\bullet$}\put(23,2){$\bullet$}\put(46,2){$\bullet$}\put(69, 2){$\bullet$}\put(92, 2){$\bullet$}\put(115, 2){$\bullet$}\put(132,2){$\bullet$}
\put(0,-5){$_1$}\put(23,-5){$_2$}\put(46,-5){$_3$}\put(70,-5){$_4$}\put(92,-5){$_5$}\put(115,-5){$_6$}\put(133,-5){$_7$}\qbezier(2,5)(35,30)(72,5)\qbezier(95,5)(107,25)(119,5)\qbezier(25,5)(60,30)(95,5)
\end{picture}}\end{center}

A \emph{representation by ties} of an $A_n$-partition consists simply by replacing, in the graphical representation, the points with parallel vertical lines and the arcs with horizontal dashed lines, called \emph{ties}, so that if they come from the same block they are connected; the order of the arcs (from left to right) is translated in ordering the ties from bottom to top. The ties are drawn as dashed lines to indicate that they are transparent to the lines that they do no connect. Below the representation by ties of the set partition $I$ above.\[\figureone\]

From now on we will adopt the representation by ties, since the product of set partitions is nicely reflected in terms of ties. Furthermore, this representation will be useful in the rest of the paper.

\begin{exm}
Here there is a product of two $\A_8$-partitions:\[\figuretwo\]
\end{exm} 


For every $i,j\in [n]$ with $i<j$ we denote by $\mu_{i,j}$ the $\A_n$-partition $\mu_{\{i,j\}}$. Below all the possibilities of $\mu_{i,j}$ as $A_4$-partitions.\[\figurefou\]

The elements $\mu_{i,j}$'s define a presentation  for $P_n$. More precisely, we have the following  theorem.
 
\begin{thm}[FitzGerald {\cite[Theorem 2]{Fi03}}]\label{001}
The monoid $P_n$ may be presented with generators $\mu_{i,j}$ for all $i,j\in [n]$ with $i<j$ subject to the following relations:\begin{enumerate}
\item[]{\normalfont(P1)}\quad$\mu_{i,j}^2=\mu_{i,j}$ for all $i,j\in [n]$ with $i<j$,
\item[]{\normalfont(P2)}\quad$\mu_{i,j}\mu_{r,s}=\mu_{r,s}\mu_{i,j}$ for all $i,j,r,t\in [n]$ with $i<j$ and $r<s$,
\item[]{\normalfont(P3)}\quad$\mu_{i,j}\mu_{j,k}=\mu_{i,j}\mu_{i,k}=\mu_{j,k}\mu_{i,k}$ for all $i,j,k\in [n]$ with $i<j<k$.
\end{enumerate}
\end{thm}
 
Diagrammatically, relation (P3) in the theorem above is represented as:\begin{equation}\label{004}\figuresix\end{equation} This relation implies the  following corollary.

\begin{crl}\label{010}
Let $r_1,\ldots,r_k\in[n]$ for some $k\geq2$ with $r_i\neq r_j$ for all $i\neq j$. Then:\[(\mu_{r_1,r_2}\cdots\mu_{r_{k-1},r_k})\mu_{i,j}=\left\{\begin{array}{ll}
\mu_{r_1,r_2}\cdots\mu_{r_{k-1},r_k}&\text{if }i=r_1\text{ and }j=r_k\\
\mu_{i,r_1}(\mu_{r_1,r_2}\cdots\mu_{r_{k-1},r_k})&\text{if }i<r_1\text{ and }j=r_k\\
(\mu_{r_1,r_2}\cdots\mu_{r_{k-1},r_k})\mu_{r_k,j}&\text{if }i=r_1\text{ and }j>r_k\\
\mu_{r_1,i}\mu_{i,r_2}\cdots\mu_{r_{k-1},r_k}&\text{if }r_1<i<r_2\text{ and }j=r_k\\
\mu_{r_1,r_2}\cdots\mu_{r_{k-1},j}\mu_{j,r_k}&\text{if }i=r_1\text{ and }r_{k-1}<j<r_k.
\end{array}\right.\]
\end{crl}

The following result is a consequence of Proposition \ref{000}, Theorem \ref{001}  and Corollary \ref{010}.

\begin{pro}[Normal form]\label{074}
Every element of $P_n$ has a unique decomposition  up to permutation of its blocks. More precisely, if $I=\{I_1,\ldots,I_k\}\in P_n$, then\[I=\mu_{I_1}\cdots\mu_{I_k}\quad\text{with}\quad\mu_{I_i}=\mu_{a_{i_1},a_{i_2}}\cdots\mu_{a_{i_{t-1}},a_{i_t}}\]where $I_i=\{a_{i_{1}}<\cdots<a_{i_t}\}$. If $I$ is an ordered set partition, the decomposition, keeping the order of the blocks, is uniquely determined.
\end{pro}

For instance, the normal form of the ordered $\A_7$-partition $I=\left(\{5\},\{1,3,7\},\{2,4,6\}\right)$ is $I=\mu_{1,3}\mu_{3,7}\mu_{2,4}\mu_{4,6}.$\[\figuresev\]

\subsubsection{Canonical system of generators}

Here we will show that relations (P1) and (P3) of Theorem \ref{001} determinate a  canonical system of generators which will be used in the next section. To be precisely, put $U_n:=\{\mu_{i,j}\mid1\leq i<j\leq n\}$  and denote by $S$  the subset of $CU_n^*\times CU_n^*$ formed by the pairs below, that define the relations (P1) and (P3) in Theorem \ref{001}.
\begin{itemize}
\item $(\mu_{i,j}^2,\mu_{i,j})$ for all $i,j\in [n]$ with $i<j$ and $|i|\leq j$,\label{088}
\item $(\mu_{i,j}\mu_{i,k},\mu_{i,j}\mu_{j,k})$ for all $i,j,k\in[n]$ with $i<j<k$,\label{075}
\item $(\mu_{i,k}\mu_{j,k},\mu_{i,j}\mu_{j,k})$ for all $i,j,k\in[n]$ with $i<j<k$,\label{086}
\end{itemize}

Note that $S$ is reduced and  $ P_n\simeq CU_n^*/\overline{S}$.  The rest of the subsection is dedicated to prove the following proposition.

\begin{pro}\label{092}
The set $S$ is a canonical system of generators of $\overline{S}$.
\end{pro}

Consider the total order on $U_n$ in which $\mu_{i,j}$ is smaller than $\mu_{r,s}$ if either $j<s$ or $j=s$ with $i>r$. As in Subsection \ref{095}, we obtain an admissible order $<$ on $CU_n^*$. In particular, we have the following inequalities:\[\mu_{i,j}\mu_{j,k}<\mu_{i,j}\mu_{i,k}<\mu_{i,k}\mu_{j,k},\qquad i,j,k\in[n],\,i<j<k.\]

For instance, we have the following inequalities:\[\figureeig\]

Informally, if we  would like to compare two elements of $U_n$ we look at their representations by ties: if both ties end at the same thread to the right, then it is bigger one that whose tie is longuer, otherwise, it is bigger one that whose tie ends more to the right.
 
Recall that $N:CU_n^*\to CU_n^*$ denotes the normal form of $CU_n^*$ respect to $S$.

\begin{lem}\label{089}
For every subset $X$ of $[n]$, the connected normal form of $\mu_X$ coincides with the $S$--normal form of every word representative of $\mu_X$ in $CX^*$.
\end{lem}
\begin{proof}
Let $u'\in CX^*$ be a word representative of $\mu_X$. Then none of the words $\mu_{i,j}^2,\allowbreak\mu_{i,j}\mu_{i,k},\mu_{i,k}\mu_{j,k}$ is a subword of $N(u')$ for all $i<j<k$. Let $u=\mu_{a_1,a_2}\mu_{a_2,a_3}\cdots\mu_{a_{m-1},a_m}$ be the longest connected normal form that is a subword of $N(u')$. Suppose that $N(u')u^{-1}\neq 1$. Since $\mu_X$ is connected, there is $\mu_{i,j}\in U_n$ such that $\mu_{i,j}|N(u')u^{-1}$ and $a_k\in\{i,j\}$ for some $k\in[m]$. If $a_k=i$ and $k<m$, then $\mu_{i,j}\mu_{i,a_{k+1}}|u'$ which is a contradiction. If $a_k=j$ and $j>1$, then $\mu_{i,j}\mu_{a_{k-1},j}|u'$ which is a contradiction as well. Otherwise, either $u\mu_{i,j}$ or $\mu_{i,j}u$ is a connected normal form, which contradicts the fact that $u$ is longest. Therefore $N(u')u^{-1}=1$.
\end{proof}

Having in mind that $S$ is defined by connected words, we deduce the following lemma.

\begin{lem}\label{090}
Let $X,Y$ be two disjoint subsets of $[n]$, and let $u,u'\in CX^*$ be word representatives of $\mu_X$ and $\mu_Y$ respectively. Then $N(uu')=N(u)N(u')$.
\end{lem}

\begin{proof}[Proof of Proposition \ref{092}]
Let $u,u'\in CU_n^*$ such that $u,u'$ are $S$-equivalent. Proposition \ref{074} implies that $u,u'$ have the same normal form. So, by Lemma \ref{090}, we obtain $N(u)=N(u')$.
\end{proof}

\subsection{ The monoid $SP_n$}

In what follows, we denote by $[\pm n]$ the set $\{\pm1,\ldots,\pm n\}$ which is considered with the usual order.  We simply write $P_{\pm n}$ instead of $P([\pm n])$ 
 and for every subset $K$ of $[\pm n]$ we denote by $-K$ the subset obtained by swapping $+$ and $-$ in $K$. A set partition $I$ of $[\pm n]$ is said to be a \emph{signed partition} if for every block $K$ of $I$, the set $-K$ is also a block of $I$. The set formed by all the signed partitions in $P_{\pm n}$ is denoted by $SP_n$.
\begin{pro}\label{066}
The set of signed partitions $SP_n$ is a submonoid of $P_{\pm n}$.
\end{pro}
\begin{proof}
Let $I,J$ be two signed partitions, let $H$ be a block of $I$, and let $K$ be a block of $J$. Assume that $H$ intersects $K$. If $H,K$ are both non zero blocks, then $-H$ intersects $-K$ with $(H\cup K)\cap(-H\cup-K)=\emptyset$. If $H$ is a zero block and $K$ is a non zero block, then $-K$ intersects $H$ such that $-K\cup K\cup H$ is a zero subset of $[\pm n]$. If $H,K$ are both zero blocks, then $H\cup K$ is a zero subset. This implies that $IJ$ is a signed partition. Therefore $SP_n$ is a submonoid of $P_{\pm n}$.
\end{proof}

Our next purpose is to describe $SP_n$ through a presentation; to do that we need to introduce first some notations. For every nonempty subset $X$ of $[\pm n]$ we will denote by $\e_X$ the set partition $\mu_{-X}\mu_X$. Note that $\e_X=\e_{-X}$. Moreover, if $X$ intersects $-X$ then $\e_X=\e_{-X\,\cup\,X}$. For instance, $\{-2,3,5\}$ is a subset of $[\pm 5]$, then $\e_{\{-2,3,5\}}$ is represented as follows\[\figurenin\]

For every $i,j\in[\pm n]$ with $i<j$ we set $\e_{i,j}=\e_{\{i,j\}}$. Note that $\e_{i,j}=\e_{-j,-i}$ for all $i<j$. For instance, in $P_{\pm\,3}$ there are 9 $\e_{i,j}$'s:\[\figureten\]

\begin{lem}\label{069}
The set $E_n:=\{\e_{i,j}\mid i,j\in[\pm n],\,i<j\}$ generates $SP_n$. Furthermore $|E_n|=n^2$ for all $n\geq2$.
\end{lem}
\begin{proof}
Proposition \ref{000},  Proposition \ref{074} and the fact that $\e_{i,j}=\e_{-j,-i}$ for all $i<j$ imply that for every subset $X=\{x_1<\cdots<x_k\}$ of $[\pm n]$, we have\begin{equation}\label{068}\e_X=\e_{x_1,x_2}\cdots\e_{x_{k-1},x_k}\qquad\e_{-X\cup X}=\e_{-x_1,x_1}\e_{x_1,x_2}\cdots\e_{x_{r-1},x_r}.
\end{equation}So, Proposition \ref{074} implies that $E_n$ generates $SP_n$.
\end{proof}

\begin{thm}\label{067}	
The monoid of signed partitions $SP_n$ may be presented with generators $\e_{i,j} \in E_n$ subject to the following relations:\begin{enumerate}
\item[]{\normalfont(SP1)}\quad$\e_{i,j}^2=\e_{i,j}$ for all $i,j\in [\pm n]$ with $i<j$,
\item[]{\normalfont(SP2)}\quad$\e_{i,j}\e_{r,s}=\e_{r,s}\e_{i,j}$ for all $i,j,r,s\in [\pm n]$ with $i<j$ and $r<s$,
\item[]{\normalfont(SP3)}\quad$\e_{i,j}\e_{j,k}=\e_{i,j}\e_{i,k}=\e_{i,k}\e_{j,k}$ for all $i,j,k\in[\pm n]$ with $i<j<k$,
\item[]{\normalfont(SP4)}\quad$\e_{i,j}=\e_{-j,-i}$ for all $i,j\in[\pm n]$ with $i<j$.
\end{enumerate}
\end{thm}
 
The rest of the section is dedicated to prove Theorem \ref{067}.

Let $Q$ be the subset of $P_{\pm n}\times P_{\pm n}$ formed by the pairs $(\mu_{i,j},\mu_{-j,-i})$ with $i,j\in[\pm n]$ and $i<j$. For every $I,J\in P_{\pm n}$ we shall denote $I\equiv J$ for $(I,J)\in\overline{Q}$. As a consequence of relation (P3) in $P_{\pm n}$ we obtain the following $Q$-equivalences:\begin{equation}\label{078}\mu_{-i,i}\mu_{i,j}\equiv\mu_{-i,i}\mu_{-i,j}\equiv\mu_{-j,j}\mu_{i,j}\equiv\mu_{-j,j}\mu_{-i,j}\equiv\mu_{-i,j}\mu_{i,j}.\end{equation}

\begin{lem}\label{082}
Let $X$ be a nonempty subset of $[\pm n]$ such that $-X$ intersects $X$. Then, the set partition $\mu_X$ is $Q$-equivalent to $\mu_Y$ where $Y=\{-\min(|X|)\}\cup|X|$.
\end{lem}
\begin{proof}
Assume that $X=\{-y_t<\cdots<-y_1<x_1<\cdots<x_k\}$ for some $x_1,\ldots,x_k,\allowbreak y_1,\ldots,y_t\in[n]$. Then\[\begin{array}{rcl}\mu_X
&=&\mu_{-y_t,-y_{t-1}}\cdots\mu_{-y_2,-y_1}\mu_{-y_1,x_1}\mu_{x_1,x_2}\cdots\mu_{x_{k-1},x_k},\\
&\equiv&\mu_{-y_1,x_1}\mu_{x_1,x_2}\cdots\mu_{x_{k-1},x_k}\mu_{y_1,y_2}\cdots\mu_{y_{t-1},y_t}.
\end{array}\]Since $-X$ intersects $X$, then $x_i=y_j$ for some $i,j\in[n]$. Hence $\mu_X\equiv\mu_{-y_1,x_1}\mu_{|X|}$. Equation \ref{078} implies that $\mu_{-y_1,x_1}\mu_{|X|}\equiv\mu_{-i,i}\mu_{|X|}=\mu_Y$ with $i=\min\{x_1,y_1\}$.
\end{proof}

\begin{lem}\label{076}
Let $X,Y$ be two subsets of $[\pm n]$ such that $\e_X=\e_Y$. Then $\mu_X$ and $\mu_Y$ are $Q$-equivalent.
\end{lem}
\begin{proof}
Since $\e_X$ and $\e_Y$ share their blocks, then $-X\cap X$ is empty if and only if $-Y\cap Y$ is empty. If $-X\cap X$ is empty, then $X\in\{-Y,Y\}$ which implies that $\mu_X$ and $\mu_Y$ are $Q$-equivalent. Assume now that $-X\cap X$ is nonempty. Since $-X\cup X=-Y\cup Y$, Lemma \ref{082} implies that $\mu_X$ and $\mu_Y$ are $Q$-equivalent.
\end{proof}

Let $s:P_{\pm n}\to SP_n$ be the epimorphism defined by $s(\mu_{i,j})=\e_{i,j}$ for all $i<j$.

\begin{lem}\label{077}
We have $\ker(s)=\overline{Q}$.
\end{lem}
\begin{proof}
Clearly $(\mu_{i,j},\mu_{-j,-i})\in\ker(s)$ because $s(\mu_{i,j})=\e_{i,j}=\e_{-j,-i}=s(\mu_{-j,-i})$ for all $i,j\in[\pm n]$ with $i<j$. Let $I=\{I_1,\ldots,I_p\}$ and $J=\{J_1,\ldots,J_q\}$ such that $(I,J)\in\ker(s)$, that is $s(I)=\mu_{I_1}\mu_{-I_1}\cdots\mu_{I_p}\mu_{-I_p}=\mu_{J_1}\mu_{-J_1}\cdots\mu_{J_p}\mu_{-J_p}=s(J)$. Since blocks in a set partition are disjoint, Proposition \ref{074} implies that $p=q$ and that for every $k\in[p]$ there is $t=t(k)\in[q]$ such that $\mu_{I_k}\mu_{-I_k}=\mu_{J_t}\mu_{-J_t}$. So, by Lemma \ref{076}, the set partitions $\mu_k$ and $\mu_{J_t}$ are $Q$-equivalent for all $k\in[p]$. This implies that $I,J$ are $Q$-equivalent. Therefore $Q$ generates $\ker(s)$.
\end{proof}

\begin{proof}[Proof of Theorem \ref{067}] 
It is a consequence of Proposition \ref{105} applied to the epimorphism $s$ and Lemma \ref{077}.
\end{proof}

\section{Monoids  of set partitions of types B and D}\label{S3}

It is well-known that the lattice of set partitions of $[n]$ is isomorphic to the intersection lattice for the hyperplane arrangement of the Coxeter group of type $\A_n$. In this context, Reiner has defined the set partitions of types B and D, which are the definitions that we shall use here, see \cite{Re97}. These partitions of types B and D turn out to be subsets of $P_{\pm n}$, however these subsets are not submonoids of $P_{\pm n}$. Thus,  in order to obtain monoid structures, we introduce a new product on the these subsets  and  so we construct presentations for them.

\subsection{The monoid $P_n^B$}

A subset $K$ of $[\pm n]$ is said to be a \emph{zero subset} (resp. \emph{non zero subset}), if $-K=K$ (resp. $K\cap -K$ is empty). A block of a signed partition that is a zero subset (resp. non zero subset) is called a \emph{zero block} (resp. \emph{non zero block}) of it. Note that every block of a signed partition is either a zero or a non zero block.

As was defined by Reiner \cite[Section 2]{Re97}, a \emph{set partition of type $\B_n$} or simply a $\B_n$-\emph{partition} is a signed partition of $[\pm n]$ having at most one zero block. We will denote by $P_n^B$ the set formed by all the $\B_n$-partitions. We have the following poset inclusions:\[P_n^B\subset SP_n\subset P_{\pm n}.\] 

Observe that the monoid structure of $P_{\pm n}$ is not transferred to $P_n^B$: below two set $\B_6$-partitions whose product is not a $\B_6$-partition.\[\figureele\]

In order to solve this blockage we need to introduce the zero closure. The \emph{zero closure} of a signed partition $I$, denoted by $I_0$, is the collection obtained by replacing all its zero subsets by the union of them. For instance:\[\figuretwe\]Note that $(I_0)_0=I_0$.

\begin{dfn}[ Product]
Given $I$ and $J$ in $P_n^B$, the product of $I$ with $J$, denoted by $I\cdot_BJ$, is defined to be the zero closure of $IJ$.
\end{dfn}

 With this product, $P_n^B$ becomes in an idempotent commutative monoid with $1=\{\pm\{1\},\ldots,\pm\{n\}\}$. Note that $P_n^B$ is generated by $E_n$ as well.

The next proposition shows how to compute the product $\cdot_B$ through $\preccurlyeq$.

\begin{pro}\label{085}
For every $\B_n$-partitions $I,J$, we have:\[I\cdot_BJ=\min\{K\in P_n^B\mid I\preccurlyeq K\text{ and }J\preccurlyeq K\}.\]
\end{pro}
\begin{proof}
Proposition \ref{066} says that $IJ$ is a quasi $\B_n$-partition. Then $I,J\preccurlyeq IJ\preccurlyeq(IJ)_0$. Hence $I\cdot_BJ\preccurlyeq(IJ)_0$. Let $E$ be a block of $(IJ)_0$. For every $a\in E$, $E$ contains the unique blocks $I_a$ of $I$ and $J_a$ of $J$ that contain $a$. This implies that every block of $(IJ)_0$ is contained in a block of $I\cdot_BJ$. Therefore $(IJ)_0\preccurlyeq I\cdot_BJ$.
\end{proof}

\begin{ntn}
If there is no risk of confusion, in what follows of this subsection we shall simply denote $I\cdot_BJ$ by $IJ$ instead, whenever $I,J\in P_n^B$.
\end{ntn}

Set $\ve_X:=(\e_X)_0$ and $\ve_{i,j}:=(\e_{i,j})_0$ for all $i<j$. Our goal now is to prove the following theorem.

\begin{thm}\label{087}	
The monoid of $\B_n$-partitions $P_n^B$ may be presented with generators $\ve_{i,j}$ with $i,j\in [\pm n]$ satisfying $i<j$, subject to the following relations:\begin{enumerate}
\item[]{\normalfont(PB1)}\quad$\ve_{i,j}^2=\ve_{i,j}$ for all $i,j\in [\pm n]$ with $i<j$,
\item[]{\normalfont(PB2)}\quad$\ve_{i,j}\ve_{r,s}=\ve_{r,s}\ve_{i,j}$ for all $i,j,r,s\in [\pm n]$ with $i<j$ and $r<s$,
\item[]{\normalfont(PB3)}\quad$\ve_{i,j}\ve_{j,k}=\ve_{i,j}\ve_{i,k}=\ve_{i,k}\ve_{j,k}$ for all $i,j,k\in[\pm n]$ with $i<j<k$,
\item[]{\normalfont(PB4)}\quad$\ve_{-i,i}\ve_{-j,j}=\ve_{-i,j}\ve_{i,j}$ for all $i,j\in[n]$ with $i<j$,
\item[]{\normalfont(PB5)}\quad$\ve_{i,j}=\ve_{-j,-i}$ for all $i,j\in[\pm n]$ with $i<j$.
\end{enumerate}
\end{thm}

The rest of the subsection is dedicated to prove Theorem \ref{087}. 

Let $T$ be the subset of $SP_n\times SP_n$ formed by the pairs $(\e_{-i,i}\e_{-j,j},\e_{-i,i}\e_{i,j})$ with $i,j\in[n]$ and $i<j$. For every $I,J\in SP_n$ we will simply denote $I\equiv J$ instead of $(I,J)\in\overline{T}$.

\begin{lem}\label{084}
Let $H,K$ be two disjoint zero subsets of $[\pm n]$. Then $\e_H\e_K$ and $\e_{K\cup H}$ are $T$-equivalent.
\end{lem}
\begin{proof}
Lemma \ref{082} implies that there are $a_1<\ldots<a_p$ and $b_1<\ldots<b_q$ in $[n]$ such that $\e_H=\e_{-a_1,a_1}\e_{a_1,a_2}\cdots\e_{a_{p-1},a_p}$ and $\e_K=\e_{-b_1,b_1}\e_{b_1,b_2}\cdots\e_{b_{q-1},b_q}$. Hence $\e_H\equiv\e_{-a_1,a_1}\cdots\e_{-a_p,a_p}$ and $\e_K\equiv\e_{-b_1,b_1}\cdots\e_{-b_q,b_q}$. Then $\e_H\e_K\equiv\e_{-c_1,c_1}\cdots\e_{-c_{p+q},c_{p+q}}$, where $H\cup K=\{\pm c_1,\ldots,\pm c_{p+q}\}$ with $c_1<\cdots<c_{p+q}$ in $[n]$. By rewriting again we obtain $\e_H\e_K\equiv\e_{-c_1,c_1}\e_{c_1,c_2}\cdots\e_{c_{p+q-1},c_{p+q}}=\e_{H\cup K}$.
\end{proof}

\begin{lem}\label{083}
We have $(IJ)_0=I_0J_0$ for all $I,J\in SP_n$. Furthermore\[\overline{T}=\{(I,J)\in SP_n\times SP_n\mid I_0=J_0\}.\]
\end{lem}

\begin{proof}
By definition $1_0=1$ and $I_0=J_0$ for all $(I,J)\in R$. Then $I\to I_0$ is an homomorphism, that is: $(IJ)_0=I_0J_0$ for all $I,J\in SP_n$. Since $(\e_{-i,i}\e_{-j,j})_0=(\e_{-i,i}\e_{i,j})_0$ for $i,j\in[n]$ with $i<j$, then $\overline{T}\subseteq\{(I,J)\mid I_0=J_0\}$. Let $I,J$ be two signed partitions such that $I_0=J_0$. Then $I,J$ share their non zero blocks and the union of their zero blocks coincide. So, Lemma \ref{084} implies that $I,J$ are $T$-equivalent.
\end{proof}

\begin{proof}[Proof of Theorem \ref{087}]
Since $E_n$ generates $P_n^B$, Lemma \ref{083} implies that the map $I\to I_0$ is an epimorphism from $SP_n$ to $P_n^B$ with kernel congruence $\overline{T}$. Hence, from Proposition \ref{105}, $P_n^B\simeq SP_n/\overline{T}$.
\end{proof}

\subsection{The monoid $P_n^D$}

In this subsection we not only introduce a monoid of set partitions of type $\D$ but also the monoid of restricted signed set partitions.

A zero subset is said to be \emph{positive} if it has more than two elements. A signed partition of $SP_n$ is said to be \emph{restricted} if all its zero blocks, if present, are positive. The set formed by all the restricted signed partitions in $SP_n$ is denoted by $RSP_n$.

As was defined by Reiner \cite{Re97}, a \emph{set partition of type $\D_n$} or simply a $\D_n$\emph{-partition} is a $\B_n$-partition in which its zero block, if present, is positive. In particular, $\e_{-i,i}$ is not a $\D_n$-partition for all $i\in[n]$. We will denote by $P_n^D$ the set of all $\D_n$-partitions. Note that we have the following poset inclusions:\[P_n^D\subset RSP_n\subset SP_n \subset P_{\pm n}\simeq P_{2n},\qquad P_n^D\subset P_n^B.\]

\subsubsection{The monoid $RSP_n$}

\begin{lem}
The set of restricted signed partitions $RSP_n$ is a submonoid of $SP_n$ generated by $E_n^{\times}:=\{\e_{i,j}\mid i,j\in[\pm n],\,i<j,\,j\neq-i\}$. Furthermore $|E_n^{\times}|=n^2-n$.
\end{lem}
\begin{proof}
It is a consequence of Lemma \ref{069} and the fact that $\e_{-i,i}\e_{i,j}=\e_{-i,j}\e_{i,j}$.
\end{proof}

Let $RP_n$ be the submonoid of $P_{\pm n}$ generated by $U_{\pm n}^{\times}:=U_{\pm n}\backslash\{\e_{-i,i}\mid i\in[n]\}$.

\begin{pro} 
The monoid $RP_n$ may be presented with generators $\mu_{i,j}$ with $i,j\in[\pm n]$ where $i<j$ and $j\neq-i$, subject to the following relations:\begin{enumerate}
\item[]{\normalfont(RP1)}\quad$\mu_{i,j}^2=\mu_{i,j}$ for all $i,j\in [\pm n]$ with $i<j$ and $j\neq-i$,
\item[]{\normalfont(RP2)}\quad$\mu_{i,j}\mu_{r,s}=\mu_{r,s}\mu_{i,j}$ for all $i,j,r,s\in [\pm n]$ with $i<j\neq-i$ \mbox{and $r<s\neq-r$},
\item[]{\normalfont(RP3)}\quad$\mu_{i,j}\mu_{j,k}=\mu_{i,j}\mu_{i,k}=\mu_{i,k}\mu_{j,k}$ for all $i,j,k\in[\pm n]$ with $i<j<k$ where $j\neq-i$ and $k\neq-j$,
\end{enumerate}
\end{pro}
\begin{proof}
It is a consequence of Proposition \ref{092} and Theorem \ref{093}.
\end{proof}

\begin{thm}\label{107} 
The monoid $RSP_n$ may be presented with generators $\e_{i,j}$ with $i,j\in[\pm n]$ where $i<j$ and $j\neq-i$, subject to the following relations:\begin{enumerate}
\item[]{\normalfont(RSP1)}\quad$\e_{i,j}^2=\e_{i,j}$ for all $i,j\in [\pm n]$ with $i<j$ and $j\neq-i$,
\item[]{\normalfont(RSP2)}\quad$\e_{i,j}\e_{r,s}=\e_{r,s}\e_{i,j}$ for all $i,j,r,s\in [\pm n]$ with $i<j\neq-i$ \mbox{and $r<s\neq-r$},
\item[]{\normalfont(RSP3)}\quad$\e_{i,j}\e_{j,k}=\e_{i,j}\e_{i,k}=\e_{i,k}\e_{j,k}$ for all $i,j,k\in[\pm n]$ with $i<j<k$ where $j\neq-i$ and $k\neq-j$,
\item[]{\normalfont(RSP4)}\quad$\e_{i,j}=\e_{-j,-i}$ for all $i,j\in[\pm n]$ where $i<j$ and $j\neq-i$.
\end{enumerate}
\end{thm}
\begin{proof}
Recall that $s:P_{\pm n}\to SP_n$ is the homomorphism defined by $s(\mu_{i,j})=\e_{i,j}$ for all $i<j$. Note that $s(RP_n)=RSP_n$ and the restriction of $s$ to $RP_n$ is surjective. Lemma \ref{077} implies that $\ker(s)=\overline{Q}\subset RP_n$. So, Proposition \ref{105} says that $RP_n/\overline{Q}\simeq RSP_n$. 
\end{proof}

\begin{rem}\label{073}
As in Remark \ref{080} we can reduce the number of generators $\e_{i,j}$ to those additionally satisfying $|i|\leq j$ and remove relation (RSP4). In this case, relation (RSP3) should be replaced by the following four relations.
\begin{enumerate}
\item[]{\normalfont(RSP3a)\!\!\!}\quad$\e_{i,j}\e_{j,k}=\e_{i,j}\e_{i,k}=\e_{i,k}\e_{j,k}$ for all $i,j,k\in[n]$ with $i<j<k$,
\item[]{\normalfont(RSP3b)\!\!\!}\quad$\e_{-i,j}\e_{j,k}=\e_{-i,j}\e_{-i,k}=\e_{-i,k}\e_{j,k}$ for all $i,j,k\in[n]$ with $i<j<k$, 
\item[]{\normalfont(RSP3c)\!\!\!}\quad$\e_{-i,j}\e_{i,k}=\e_{-j,k}\e_{i,k}=\e_{-i,j}\e_{-j,k}$ for all $i,j,k\in[n]$ with $i<j<k$,
\item[]{\normalfont(RSP3d)\!\!\!}\quad$\e_{i,j}\e_{-i,k}=\e_{i,j}\e_{-j,k}=\e_{-j,k}\e_{-i,k}$ for all $i,j,k\in[n]$ with $i<j<k$.
\end{enumerate}
\end{rem}

\subsubsection{The monoid $P_n^D$}

\begin{thm}\label{106}
The monoid of $\D_n$-partitions $P_n^D$ may be presented with generators $\ve_{i,j}$ with $i,j\in[\pm n]$ satisfying $i<j$ and $j\neq-i$, subject to the following relations:\begin{enumerate}
\item[]{\normalfont(PD1)}\quad$\ve_{i,j}^2=\ve_{i,j}$ for all $i,j\in [\pm n]$ with $i<j$ and $j\neq-i$,
\item[]{\normalfont(PD2)}\quad$\ve_{i,j}\ve_{r,s}=\ve_{r,s}\ve_{i,j}$ for all $i,j,r,s\in [\pm n]$ with $i<j\neq-i$ \mbox{and $r<s\neq-r$},
\item[]{\normalfont(PD3)}\quad$\ve_{i,j}\ve_{j,k}=\ve_{i,j}\ve_{i,k}=\ve_{i,k}\ve_{j,k}$ for all $i,j,k\in[\pm n]$ with $i<j<k$ where $j\neq-i$ and $k\neq-j$,
\item[]{\normalfont(PD4)}\quad$\ve_{i,j}=\ve_{-j,-i}$ for all $i,j\in[\pm n]$ where $i<j$ and $j\neq-i$,
\item[]{\normalfont(PD5)}\quad$\ve_{-i,j}\ve_{i,j}\ve_{-r,s}\ve_{r,s}=\ve_{-a,b}\ve_{a,b}\ve_{b,c}\ve_{c,d}$ for all $i,j,r,s,a,b,c,d\in[\pm n]$ satisfying $\{i<j,r<s\}=\{a<b<c<d\}$.
\end{enumerate}
\end{thm}
\begin{proof}
Let $F$ be the subset of $RSP_n\times RSP_n$ formed by the pairs $(\e_{-i,j}\e_{i,j}\e_{-r,s}\e_{r,s},\allowbreak\e_{-a,b}\e_{a,b}\e_{b,c}\e_{c,d})$ with $i,j,r,s\in[n]$ satisfying $\{i<j,r<s\}=\{a<b<c<d\}$. By Lemma \ref{083}, $\overline{F}$ is the kernel congruence of the homomorphism $I\to I_0$ restricted to $RSP_n$. So, Proposition \ref{105} implies that $P_n^D\simeq RSP_n/\overline{F}$.
\end{proof} 

\section{Tied monoids of type A}\label{S4}

We will study in this section the tied monoid $T\!M = P\rtimes_{\rho}M$ introduced in Definition \ref{155}, where  $P=P_n$ and $M$ is the braid group or more generall a braid--like monoid. In order to describe the action $\rho$ we need to introduce first some notations. As usual, we denote by $\SS_n$ the \emph{symmetric group} on $n$ symbols. In what follows $M$ denotes a monoid accepting an epimorphism $\pi:M\to\SS_n$ and denote by $\pp$ the \emph{permutation action} of $\SS_n$ on $[n]$. This one induces an action, denoted again by $\pp$, of $\SS_n$ on $P_n$, that is,\begin{equation}\label{135}\pp_g(I)=\{\pp_g(I_1),\ldots,\pp_g(I_k)\},\qquad g\in\SS_n,\,I=\{I_1,\ldots,I_k\}\in P_n.\end{equation}
So, the action $\rho:M\to\End(P_n)$ is given by $\rho:=\pp\circ\pi$.

We organize the section in two subsections. The first one gives the necessary notations, as well technical results related with the symmetric group and the generators $\mu_{i,j}$'s all which will be used  in the second subsections and next Section; in the second one we will write out the presentation of some $TM$'s.
 
\subsection{A--Preliminaries} 

For $i\in[n-1]$ denote by $s_i$ the transposition exchanging $i$ with $i+1$. Recall that $\SS_n$ is generated by transpositions and may be presented with generators $s_1,\ldots,s_{n-1}$ subject to the following relations:\begin{enumerate}
\item[](A1)\quad$s_is_js_i=s_js_is_j$ for all $i,j\in[n-1]$ with $|i-j|=1$,
\item[](A2)\quad$s_is_j=s_js_i$ for all $i,j\in[n-1]$ with $|i-j|\geq2$,
\item[](A3)\quad$s_i^2=1$ for all $i\in[n-1]$.
\end{enumerate}Recall that $\SS_n$ acts on $P_n$ by $s(I):=\{s(I_1),\ldots,s(I_k)\}$ for all $s\in\SS_n$ and $I=\{I_1,\ldots,I_k\}\in P_n$. For $i,j,k\in[n]$ with $i<j$, we have\begin{equation}\label{100}s_k(\mu_{i,j})=\left\{\begin{array}{ll}
\mu_{i+1,j}&\text{if }i=k\text{ and }k+1<j,\\
\mu_{i,j-1}&\text{if }i<k\text{ and }j=k+1,\\
\mu_{i-1,j}&\text{if }i=k+1\text{ and }k+1<j,\\
\mu_{i,j+1}&\text{if }i<k\text{ and }j=k,\\
\mu_{i,j}&\text{if }(i,j)=(k,k+1)\text{ or }i,j\not\in\{k,k+1\}.
\end{array}\right.\end{equation}

Let $\langle A\mid R\rangle$ be a presentation of $M$, and let $x_1,\ldots,x_{n-1}\in A$ such that $\rho_{x_i}=\pi_{s_i}$ for all $i\in[n-1]$. We will study the relations of $TM$ obtained through the Lavers' method. For that we set $S$ to be the set of pairs $(\mu_{i,j}x_k,x_k\rho_{x_k}(\mu_{i,j}))$ in $TM\times TM$ for all $i,j\in[n]$ with $i<j$ and $k\in[n-1]$.

\begin{lem}\label{101}
The set $S$ corresponds to the following relations of $TM$:
\begin{enumerate}
\item $\mu_{i,j}x_i=x_i\mu_{i+1,j}$ if $i+1<j$,
\item $\mu_{i,j+1}x_j=x_j\mu_{i,j}$ if $i<j$,
\item $\mu_{i+1,j}x_i=x_i\mu_{i,j}$ if $i+1<j$,
\item $\mu_{i,j}x_j=x_j\mu_{i,j+1}$ if $i<j$,
\item $\mu_{i,j}x_k=x_k\mu_{i,j}$ if $(i,j)=(k,k+1)$ or $i,j\not\in\{k,k+1\}$.
\end{enumerate}
\end{lem}
\begin{proof}
Since $\rho_{x_i}=\pi_{s_i}$ for all $i\in[n-1]$, it is obtained directly by applying \ref{100}.
\end{proof}

The previous lemma implies the following corollary.

\begin{crl}\label{097}
Assume that $x_i$ is invertible in $M$ for all $i\in[n-1]$. Then $S$ corresponds to the following relations of $TM$.\begin{enumerate}
\item $\mu_{i,j+1}=x_j\mu_{i,j}x_j^{-1}=x_j^{-1}\mu_{i,j}x_j$ if $i<j$,
\item$\mu_{i,j}=x_i^{-1}\mu_{i+1,j}x_i=x_i\mu_{i+1,j}x_i^{-1}$ if $i+1<j$,
\item$\mu_{i,j}x_k=x_k\mu_{i,j}$ if $(i,j)=(k,k+1)$ or $i,j\not\in\{k,k+1\}$.
\end{enumerate}
\end{crl}

In what follows, we assume that $x_1,\ldots,x_{n-1}$ are invertible in $M$.

For $i\in[n-1]$ we set $\eta_i:=\mu_{i,i+1}$ and add these symbols to the presentation of $TM$ by using Tietze transformations of type \tf{3}. Note that, by definition, $\eta_1,\ldots,\eta_{n-1}$ satisfy relations (P1) and (P2) of $P_n$.

For $i,j\in[n]$ with $i<j$ set:\begin{equation}\label{029}\begin{array}{ll}
a_{i,j}=\left\{\begin{array}{ll}
1&\text{if }j=i+1,\\
x_i\cdots x_{j-2}&\text{if }j>i+1,
\end{array}\right.
&\quad
b_{i,j}=\left\{\begin{array}{ll}
1&\text{if }j=i+1,\\
x_{j-1}\cdots x_{i+1}&\text{if }j>i+1,
\end{array}\right.
\\[0.8cm]
\oa_{i,j}=\left\{\begin{array}{ll}
1&\text{if }j=i+1,\\
x_{j-2}\cdots x_i&\text{if }j>i+1,
\end{array}\right.
&\quad
\ob_{i,j}=\left\{\begin{array}{ll}
1&\text{if }j=i+1,\\
x_{i+1}\cdots x_{j-1}&\text{if }j>i+1.
\end{array}\right.\end{array}\end{equation}

\begin{lem}\label{098}
For each $i,j\in[n]$ with $i<j$ we have the following:
\begin{enumerate}
\item$\mu_{i,j}=\oa_{i,j}^{-1}\eta_{j-1}\oa_{i,j}=a_{i,j}\eta_{j-1}a_{i,j}^{-1}$,\label{033}
\item$\mu_{i,j}=\ob_{i,j}^{-1}\eta_i\ob_{i,j}=b_{i,j}\eta_ib_{i,j}^{-1}$.
\end{enumerate}
\end{lem}
\begin{proof}
It is obtained by applying Corollary \ref{097} recursively.
\end{proof}

\begin{crl}\label{099}
The set $S$ corresponds to the following relations.
\begin{enumerate}
\item $\mu_{i,j}=\oa_{i,j}^{-1}\eta_{j-1}\oa_{i,j}=b_{i,j}\eta_ib_{i,j}^{-1}=a_{i,j}\eta_{j-1}a_{i,j}^{-1}=\ob_{i,j}^{-1}\eta_i\ob_{i,j}$,\label{032}
\item$(a_{i,j}\eta_{j-1}a_{i,j}^{-1})x_k=x_k(a_{i,j}\eta_{j-1}a_{i,j}^{-1})$ if $(i,j)=(k,k+1)$ or $i,j\not\in\{k,k+1\}$.\label{038}
\end{enumerate}
\end{crl}
\begin{proof}
It is a direct consequence of Lemma \ref{098} and Corollary \ref{097}.
\end{proof}

\begin{rem}\label{103} 
Lemma \ref{098}(\ref{033}) implies that each appearance of $\mu_{i,j}$ can be replaced by $a_{i,j}^{-1}\eta_{j-1}a_{i,j}$ in all relations of $TM$. In particular in relations (P1) to (P3). Hence, as in Corollary \ref{099}(\ref{032}), each $\mu_{i,j}$ appears at most once, so, Proposition \ref{012} implies that the generators $\mu_{i,j}$, with $i<j+1$, can be removed from the presentation of $TM$. Note that relations (P1) to (P3) were not removed, they appear with generators $\mu_{i,j}$ replaced by $a_{i,j}^{-1}\eta_{j-1}a_{i,j}$.  Moreover, every $\mu_{i,j}$ can be reached by $\eta_1$ via conjugation, indeed we have $\mu_{i,j}=\oa_{1,i+1}b_{1,k}\eta_1b_{1,k}^{-1}\oa_{1,i+1}^{-1}$ for all $i,j\in[n]$ with $i<j$.
\end{rem}

In what follows, we assume that the elements $x_1,\ldots,x_{n-1}$ satisfy relations (A1) and (A2) in $M$.

\begin{pro}\label{104}
Relations obtained by $S$ are equivalent to the following.
\begin{enumerate}
\item $x_ix_j\eta_i=\eta_jx_ix_j$ and $x_ix_j^{-1}\eta_i=\eta_jx_ix_j^{-1}$ for all $i,j\in[n-1]$ with $|i-j|=1$,\label{148}
\item $x_i\eta_j=\eta_jx_i$ for all $i,j\in[n-1]$ with $|i-j|\neq1$.\label{149}
\end{enumerate}
\end{pro}
\begin{proof}
By Corollary \ref{099}(\ref{032}) we obtain, for $i,j\in[n]$ with $i<j$, the following relations:\[\eta_{j-1}\oa_{i,j}b_{i,j}=\oa_{i,j}b_{i,j}\eta_i,\quad\eta_{j-1}a_{i,j}^{-1}b_{i,j}=a_{i,j}^{-1}b_{i,j}\eta_i,\quad\eta_i\ob_{i,j}a_{i,j}=\ob_{i,j}a_{i,j}\eta_{j-1};\]note that the second one is equivalent to $b_{i,j}^{-1}a_{i,j}\eta_{j-1}=\eta_ib_{i,j}^{-1}a_{i,j}$. In particular, if $j=i+1$ the relations are trivial, and if $j=i+2$ we obtain:\[\eta_{i+1}x_ix_{i+1}=x_ix_{i+1}\eta_i,\quad\eta_{i+1}x_i^{-1}x_{i+1}=x_i^{-1}x_{i+1}\eta_i,\quad\eta_ix_{i+1}x_i=x_{i+1}x_i\eta_{i+1}\]which is just (1). Similarly, if $j=i+1$ in Corollary \ref{099}(\ref{038}) we obtain $\eta_ix_k=x_k\eta_i$ with $i=k$ or $i\not\in\{k,k+1\}$, which is just (2). We will prove that the remaining relations in Corollary \ref{099} are all obtained from (1) and (2). By using (1), we have:\[\begin{array}{c}
\eta_{j-1}\oa_{i,j}b_{i,j}=\eta_{j-1}(x_{j-2}x_{j-1})\cdots(x_ix_{i+1})=\oa_{i,j}b_{i,j}\eta_i,\\[1.5mm]
\eta_{j-1}a_{i,j}^{-1}b_{i,j}=\eta_{j-1}(x_{j-2}^{-1}x_{j-1})\cdots(x_i^{-1}x_{i+1})=\oa_{i,j}b_{i,j}\eta_i,\\[1.5mm]
\eta_i\ob_{i,j}a_{i,j}=\eta_i(x_{i+1}x_i)\cdots(x_{j-1}x_{j-2})=\ob_{i,j}a_{i,j}\eta_{j-1}.
\end{array}\]Now, we use (2) to obtain relations in Corollary \ref{099}(\ref{038}). If $k<i-1$ or $k>j$ we are done because $a_{i,j}x_k=x_ka_{i,j}$. If $i<k<j-1$ we obtain\[(a_{i,j}\eta_{j-1}a_{i,j}^{-1})x_k=a_{i,j}\eta_{j-1}x_{k-1}a_{i,j}^{-1}=a_{i,j}x_{k-1}\eta_{j-1}a_{i,j}^{-1}=x_k(a_{i,j}\eta_{j-1}a_{i,j}^{-1})\]This proves that we only need (1) and (2) to capture all relations of $S$.
\end{proof}

Let $y_1,\ldots,y_{n-1}\in X$ with $y_i\neq x_j$ for all $i\neq j$ such that $\rho_{y_i}=\pi_{s_i}$ for all $i\in[n-1]$. Denote by $T$ the set of pairs $(\mu_{i,j}y_k,y_k\rho_{y_k}(\mu_{i,j}))$ for $i,j\in[n]$ with $i<j$ and $k\in[n-1]$.

\begin{lem}\label{018}
The set $T$ corresponds to the following relations:
\begin{enumerate}
\item $\eta_i\ob_{i,j}y_ib_{i+1,j}=\ob_{i,j}y_ib_{i+1,j}\eta_{i+1}$ if $i+1<j$,\label{039}
\item $\eta_j\oa_{i,j+1}y_ja_{i,j}=\oa_{i,j+1}y_ja_{i,j}\eta_{j-1}$ if $i<j$,\label{040}
\item $\eta_{i+1}\ob_{i+1,j}y_ib_{i,j}=\ob_{i+1,j}y_ib_{i,j}\eta_i$ if $i+1<j$,\label{041}
\item $\eta_{j-1}\oa_{i,j}y_ja_{i,j+1}=\oa_{i,j}y_ja_{i,j+1}\eta_j$ if $i<j$,\label{042}
\item $\eta_{j-1}\oa_{i,j}y_ka_{i,j}=\oa_{i,j}y_ka_{i,j}\eta_{j-1}$ if $(i,j)=(k,k+1)$ or $i,j\not\in\{k,k+1\}$.\label{043}
\end{enumerate}
\end{lem}
\begin{proof}
It is a consequence of Lemma \ref{101}, Lemma \ref{098} and Remark \ref{103}.
\end{proof}

\begin{pro}\label{019}
Suppose that the following relations hold in $TM$:
\begin{itemize}
\item$y_ix_jx_i=x_jx_iy_j$ for all $i,j\in[n-1]$ with $|i-j|=1$,
\item$y_jx_i=x_iy_j$ for all $i,j\in[n-1]$ with $|i-j|\neq1$.
\end{itemize}
Then the elements of $T$ are equivalent to the following relations:
\begin{enumerate}
\item$\eta_ix_jy_i=x_jy_i\eta_j$ for all $i,j\in[n-1]$ with $|i-j|=1$,\label{044}
\item$\eta_iy_jx_i=y_jx_i\eta_j$ for all $i,j\in[n-1]$ with $|i-j|=1$,\label{045}
\item$\eta_iy_j=y_j\eta_i$ for all $i,j\in[n-1]$ with $|i-j|\neq1$.\label{046}
\end{enumerate}
\end{pro}
\begin{proof}
We obtain (1) and (2) by taking $|i-j|=1$ in Lemma \ref{018}(\ref{039}--\ref{042}), and (3) by taking $j=i+1$ in Lemma \ref{018}(\ref{043}). We will show that relations of Lemma \ref{018}(\ref{039}--\ref{042}) are consequences of (1) and (2). For $i+2<j$ we have:\[\begin{array}{c}
\eta_i\ob_{i,j}y_ib_{i+1,j}=\eta_ix_{i+1}y_i\ob_{i+1,j}b_{i+1,j}=x_{i+1}y_i\eta_{i+1}\ob_{i+1,j}b_{i+1,j}=\ob_{i,j}y_ib_{i+1,j}\eta_{i+1},\\[1.5mm]
\ob_{i+1,j}y_ib_{i,j}\eta_i=\ob_{i+1,j}b_{i+1,j}y_ix_{i+1}\eta_i=\ob_{i+1,j}b_{i+1,j}\eta_{i+1}y_ix_{i+1}=\eta_{i+1}\ob_{i+1,j}y_ib_{i,j}.
\end{array}\]Similarly, for $i+1<j$, we have the following:\[\begin{array}{c}
\eta_j\oa_{i,j+1}y_ja_{i,j}=\eta_jx_{j-1}y_j\oa_{i,j}a_{i,j}=x_{j-1}y_j\eta_{j-1}\oa_{i,j}a_{i,j}=\oa_{i,j+1}y_ja_{i,j}y_j\eta_{j-1},\\[1.5mm]
\oa_{i,j}y_ja_{i,j+1}\eta_j=\oa_{i,j}a_{i,j}y_jx_{j-1}\eta_j=\oa_{i,j}a_{i,j}\eta_{j-1}y_jx_{j-1}=\eta_{j-1}\oa_{i,j}y_ja_{i,j+1}.\end{array}\]Now, we show that Lemma \ref{018}(\ref{043}) is a consequence of (3). If $k<i-1$ or $k>j$ we are done because $x_k\oa_{i,j}=\oa_{i,j}x_k$ and $x_ka_{i,j}=a_{i,j}x_k$. If $i<k<j-1$ we obtain\[\eta_{j-1}\oa_{i,j}y_ka_{i,j}=\eta_{j-1}\oa_{i,j}a_{i,j}y_{k-1}=\oa_{i,j}a_{i,j}\eta_{j-1}y_{k-1}=\oa_{i,j}a_{i,j}y_{k-1}\eta_{j-1}=\oa_{i,j}y_ka_{i,j}\eta_{j-1}.\]This proves that we only need (1) to (3) to represent all relations of $T$.
\end{proof}

\begin{rem}\label{020}
Let $z_1,\ldots,z_{n-1}\in X$ such that, either $z_i=x_i$ for all $i\in[n-1]$ or $z_i=y_i$ for all $i\in[n-1]$, and let $i<j<k$ be consecutive numbers in $[n]$. Because of the presence of (P3) in the presentation of $TM$, we obtain the following relations:\begin{equation}\label{126}\mu_{i,j}\mu_{j,k}z_i=\mu_{j,k}\mu_{i,k}z_i=\mu_{i,j}\mu_{i,k}z_i,\qquad \mu_{i,j}\mu_{j,k}z_j=\mu_{i,j}\mu_{i,k}z_j=\mu_{j,k}\mu_{i,k}z_j.\end{equation}By applying Lemma \ref{101}, once on the second term and twice on the last one, we have:\[\mu_{i,j}\mu_{j,k}z_i=\mu_{j,k}z_i\mu_{j,k}=z_i\mu_{i,j}\mu_{j,k},\qquad\mu_{i,j}\mu_{j,k}z_j=\mu_{i,j}z_j\mu_{i,j}=z_j\mu_{j,k}\mu_{i,j}.\]Since $i,j,k$ are consecutive, in $TM$ we obtain\begin{equation}
\label{150}\eta_i\eta_jz_i=\eta_jz_i\eta_j=z_i\eta_i\eta_j,\quad|i-j|=1.\end{equation}Proposition \ref{104}(2) implies that each relation of (P3) in $TM$ can be obtained from these relations. More specifically, for $i,j,k\in[n]$ with $i<j<k$, Corollary \ref{097} implies\[\mu_{i,j}\mu_{i,k}=g\mu_{1,2}\mu_{1,3}g^{-1},\quad\mu_{i,j}\mu_{i,k}=g\mu_{1,2}\mu_{2,3}g^{-1},\quad\mu_{i,j}\mu_{i,k}=g\mu_{1,3}\mu_{2,3}g^{-1}.\]where $g=\oa_{1,i+1}b_{1,j}b_{2,k}$. 
In this way, we obtain that (P3) can be completely deduced from \ref{150} whenever $z_i$ is invertible with $i\in[n-1]$.  Similarly we can prove that commutativity of $\mu$'s is derived, in $TM$, from the commutativity of $\eta$'s.
\end{rem}

\begin{rem}\label{022}
Under the hypothesis of Proposition \ref{019}, we assume that $x_iy_i=y_ix_i$ for all $i\in[n-1]$. Let $i<j<k$ be two consecutive numbers in $[n]$, then, we obtain the relations $x_iy_i\mu_{i,k}=y_ix_i\mu_{i,k}$ and $x_jy_j\mu_{i,k}=y_jx_j\mu_{i,k}$. By applying Lemma \ref{101}, once on each term, we have $x_i\mu_{j,k}y_i=y_i\mu_{j,k}x_i$ and $x_j\mu_{i,j}y_j=y_j\mu_{i,j}x_j$. Since $i,j,k$ are consecutive, we obtain $x_i\eta_jy_i=y_i\eta_jx_i$ for all $i,j\in[n-1]$ with $|i-j|=1$. This relation is equivalent to $y_i\eta_j=x_i\eta_jy_ix_i^{-1}=x_i\eta_jx_i^{-1}y_i$ with $|i-j|=1$. Furthermore, by using Proposition \ref{019}(\ref{045}), we obtain the relation $\eta_iy_jy_i=\eta_iy_jx_ix_i^{-1}y_i=y_jx_i\eta_jx_i^{-1}y_i=y_jy_i\eta_j$ for all $i,j\in[n-1]$ with $|i-j|=1$.
\end{rem}

\subsection{Examples of $TM$}

\subsubsection{Tied braid and tied singular braid monoids}

The \emph{braid group} \cite{Ar25,Ar47}, denoted by $\BB_n$, is presented with generators $\sigma_1,\ldots,\sigma_{n-1}$ subject to the following relations:\begin{enumerate}
\item[](B1)\quad$\sigma_i\sigma_j\sigma_i=\sigma_j\sigma_i\sigma_j$ for all $i,j\in [n-1]$ with $|i-j|=1$,
\item[](B2)\quad$\sigma_i\sigma_j=\sigma_j\sigma_i$ for all $i,j\in [n-1]$ with $|i-j|\geq2$.\end{enumerate}

Diagrammatically, (B1) and (B2) can be represented, respectively, as follows:\[\figuretenthr\]

There is a natural epimorphism $\pi:\BB_n\to\SS_n$ defined by $\pi_{\sigma_i}=s_i$ for all $i\in[n-1]$. Hence, we may construct a presentation, via the Lavers' method (Theorem \ref{013}), of the tied monoid $TM$ with $M=\BB_n$.

\begin{pro}\label{014}
By applying the Lavers' method the tied monoid $TM$, with $M=\BB_n$ can be presented with generators $\sigma_1,\ldots,\sigma_{n-1}$ and $\sigma_1^{-1},\ldots,\sigma_{n-1}^{-1}$ satisfying {\normalfont(B0)} to {\normalfont(B2)}, and the additional generators $\eta_1,\ldots,\eta_{n-1}$ subject to the relations {\normalfont(T1)} to {\normalfont(T5)}, where:\begin{enumerate}
\item[]{\normalfont(B0)}\quad$\sigma_i\sigma_i^{-1}=\sigma_i^{-1}\sigma_i=1$ for all $i\in [n-1]$,
\item[]{\normalfont(T1)}\quad$\eta_i^2=\eta_i$ for all $i\in [n-1]$,
\item[]{\normalfont(T2)}\quad$\eta_i\eta_j=\eta_j\eta_i$ for all $i,j\in [n-1]$,
\item[]{\normalfont(T3)}\quad$\eta_i\sigma_j\sigma_i=\sigma_j\sigma_i\eta_j$ for all $i,j\in[n-1]$ with $|i-j|=1$,
\item[]{\normalfont(T4)}\quad$\sigma_i\eta_j=\eta_j\sigma_i$ for all $i,j\in [n-1]$ with $|i-j|\neq1$,
\item[]{\normalfont(T5)}\quad$\eta_i\eta_j\sigma_i=\eta_j\sigma_i\eta_j=\sigma_i\eta_i\eta_j$ for all $i,j\in [n-1]$ with $|i-j|=1$,
\item[]{\normalfont(T6)}\quad$\eta_i\sigma_j\sigma_i^{-1}=\sigma_j\sigma_i^{-1}\eta_j$ for all $i,j\in[n-1]$ with $|i-j|=1$.
\end{enumerate}
\end{pro}
\begin{proof}
According to the Lavers' method (Theorem \ref{013}), the relations (B0) to (B2) are obtained directly from $\BB_n$, and relations (P1) to (P3) from $P_n$. (T1) and (T2) are consequences of (P1) and (P2) in $P_n$ because the generators $\mu_{i,j}$ with $i<j-1$ are removed as in Remark \ref{103}. The remaining relations are obtained from Proposition \ref{104}, that is, relations (T3), (T4) and (T6). Finally, relation (T5) is a consequence of (P3) and can be obtained as in Remark \ref{020}. By \cite[Lemma 2(5--6)]{PreAiJu18} and Remark \ref{020}, relations (P1) to (P3) can be obtained by using (T1), (T2) and (T5) respectively, so they are removed from the presentation of $TM$. Thus, the proof is concluded.
\end{proof}

Diagrammatically, (B0) and (T1) to (T5) can be represented as follows:\[\figuretenfou\]

\begin{rem}
The Lavers' presentation for  the tied monoid $TM$ with $M=\BB_n$ coincide with the defining presentation of the {\it tied braid monoid} defined in \cite{AiJu16}. It is worth to observe that the defining presentation of the tied braid monoid was motived from  purely topological reasons.
\end{rem}
 
The \emph{singular braid monoid} \cite{Ba92,Bi93,Sm88}, denoted by $S\BB_n$, is presented with generators $\sigma_1,\ldots,\sigma_{n-1}$ and $\sigma_1^{-1},\ldots,\sigma_{n-1}^{-1}$ satisfying {\normalfont(B0)} to {\normalfont(B2)}, and $\tau_1,\ldots,\tau_{n-1}$ as additional generators subject to the following relations:\begin{enumerate}
\item[](S1)\quad$\tau_i\tau_j=\tau_j\tau_i$ for all $i,j\in[n-1]$ with $|i-j|\geq2$,
\item[](S2)\quad$\tau_i\sigma_j=\sigma_j\tau_i$ for all $i,j\in[n-1]$ with $|i-j|\neq1$,
\item[](S3)\quad$\tau_i\sigma_j\sigma_i=\sigma_j\sigma_i\tau_j$ for all $i,j\in[n-1]$ with $|i-j|\geq2$.
\end{enumerate}
Diagrammatically, (S1) to (S3) can be represented, respectively, as follows:\[\figuretenfiv\]

There is a natural epimorphism $\pi:S\BB_n\to\SS_n$ defined by $\pi_{\sigma_i}=s_i$ and $\pi_{\tau_i}=s_i$ for all $i\in [n-1]$. Hence, we may construct its tied monoid, that is $TS\BB_n$.

\begin{pro}\label{021}
According to the Lavers' method, the tied monoid $TM$ with $M=TS\BB_n$ can be presented with generators $\sigma_1,\ldots,\sigma_{n-1}$ and $\sigma_1^{-1},\ldots,\sigma_{n-1}^{-1}$ satisfying {\normalfont(B0)} to {\normalfont(B2)}, the generators $\tau_1,\ldots,\tau_{n-1}$ satisfying {\normalfont(S1)} to {\normalfont(S3)}, and the generators $\eta_1,\ldots,\eta_{n-1}$ satisfying {\normalfont(T1)} to {\normalfont(T5)} and subject to the following relations:\begin{enumerate}
\item[]{\normalfont(TS1)}\quad$\tau_i\eta_j=\eta_j\tau_i$ for all $i,j\in [n-1]$ with $|i-j|\neq1$,
\item[]{\normalfont(TS2)}\quad$\eta_i\tau_j\tau_i=\tau_j\tau_i\eta_j$ for all $i,j\in [n-1]$ with $|i-j|=1$,
\item[]{\normalfont(TS3)}\quad$\eta_i\tau_j\sigma_i=\tau_j\sigma_i\eta_j$ for all $i,j\in [n-1]$ with $|i-j|=1$,
\item[]{\normalfont(TS4)}\quad$\eta_i\sigma_j\tau_i=\sigma_j\tau_i\eta_j$ for all $i,j\in [n-1]$ with $|i-j|=1$,
\item[]{\normalfont(TS5)}\quad$\tau_i\eta_j=\sigma_i\eta_j\sigma_i^{-1}\tau_i$ for all $i,j\in [n-1]$ with $|i-j|=1$,
\item[]{\normalfont(TS6)}\quad$\eta_i\eta_j\tau_i=\eta_j\tau_i\eta_j=\tau_i\eta_i\eta_j$ for all $i,j\in [n-1]$ with $|i-j|=1$.
\end{enumerate}
\end{pro} 
\begin{proof} 
The relations (B0) to (B2) and (S1) to (S3) are inherited from $S\BB_n$, and relations (P1) to (P3) from $P_n$. Relations (T1) and (T2) are consequences of (P1) and (P2) in $P_n$ because the generators $\mu_{i,j}$ with $i<j-1$ are removed as in Remark \ref{103}. The remaining relations are obtained from Proposition \ref{104} and Proposition \ref{019}, that is, relations (T3), (T4), (TS1), (TS3) and (TS4). Relation (TS2) and (TS5) are redundant and can be obtained as in Remark \ref{022}. Relation (T5) is a consequence of (P3) and can be obtained as in Remark \ref{020}. By \cite[Lemma 2(5--6)]{PreAiJu18} and Remark \ref{020}, relations (P1) to (P3) can be obtained by using (T1), (T2) and (T5) respectively, so they are removed from the Lavers' presentation of $TS\BB_n$. Finally, relation (TS6) is  redundant since is obtained from (T5) as in Remark \ref{020}.
\end{proof}

Diagrammatically, (TS1) to (TS6) can be represented, respectively, as follows:\[\figuretensix\]

\begin{rem}\label{016}
The monoid $TS\BB_n$ was first introduced and constructed in \cite{PreAiJu18} and is called the \emph{tied singular braid monoid} and, as for $T\BB_n$, the original definition of $TS\BB_n$ is through a presentation,   conceived from topological considerations, which coincides that of Proposition \ref{021}.
\end{rem}

\subsubsection{Tied virtual braid and ties virtual singular braid monoids}\label{094}

The \emph{virtual braid group} \cite{Ka99,Kd04,KaLa04}, denoted by $V\BB_n$, is presented with generators $\sigma_1,\ldots,\sigma_{n-1}$ satisfying (B1) and (B2), and generators $\nu_1,\ldots,\nu_{n-1}$ subject to the following relations:\begin{enumerate}
\item[](V1)\quad$\nu_i^2=1$ for all $i\in [n-1]$,
\item[](V2)\quad$\nu_i\nu_j\nu_i=\nu_j\nu_i\nu_j$ for all $i,j\in [n-1]$,
\item[](V3)\quad$\nu_i\nu_j=\nu_j\nu_i$ for all $i,j\in [n-1]$ with $|i-j|\geq2$,
\item[](V4)\quad$\sigma_i\nu_j\nu_i=\nu_j\nu_i\sigma_j$ for all $i,j\in [n-1]$ with $|i-j|=1$,
\item[](V5)\quad$\sigma_i\nu_j=\nu_j\sigma_i$ for all $i,j\in [n-1]$ with $|i-j|\geq2$.
\end{enumerate}

Diagrammatically, (V1) to (V5) can be represented, respectively, as follows:\[\figuretensev\]

There is a natural epimorphism $\pi:V\BB_n\to\SS_n$ defined by $\pi_{\sigma_i}=s_i$ and $\pi_{\nu_i}=s_i$ for all $i\in[n-1]$. Hence, we may construct its tied monoid, that is $TV\BB_n$. We will call $TV\BB_n$ the \emph{tied virtual braid monoid}.

\begin{thm}\label{028}
The monoid $TV\BB_n$ can be presented with generators $\sigma_1,\ldots,\sigma_{n-1}$ and $\sigma_1^{-1},\ldots,\sigma_{n-1}^{-1}$ satisfying {\normalfont(B0)} to {\normalfont(B2)}, the generators $\nu_1,\ldots,\nu_{n-1}$ satisfying {\normalfont(V1)} to {\normalfont(V5)}, and the generators $\eta_1,\ldots,\eta_{n-1}$ satisfying {\normalfont(T1)} to {\normalfont(T5)} and subject to the following relations:\begin{enumerate}
\item[]{\normalfont(TV1)}\quad$\eta_i\nu_j\nu_i=\nu_j\nu_i\eta_j$ for all $i,j\in[n-1]$ with $|i-j|=1$,
\item[]{\normalfont(TV2})\quad$\eta_i\nu_j\sigma_i=\nu_j\sigma_i\eta_j$ for all $i,j\in [n-1]$ with $|i-j|=1$,
\item[]{\normalfont(TV3)}\quad$\eta_i\sigma_j\nu_i=\sigma_j\nu_i\eta_j$ for all $i,j\in [n-1]$ with $|i-j|=1$,
\item[]{\normalfont(TV4)}\quad$\nu_i\eta_j=\sigma_i\eta_j\sigma_i^{-1}\nu_i$ for all $i,j\in [n-1]$ with $|i-j|=1$,
\item[]{\normalfont(TV5)}\quad$\eta_i\nu_j=\nu_j\eta_i$ for all $i,j\in [n-1]$ with $|i-j|\neq1$,
\item[]{\normalfont(TV6)}\quad$\eta_i\eta_j\nu_i=\eta_j\nu_i\eta_j=\nu_i\eta_i\eta_j$ for all $i,j\in [n-1]$ with $|i-j|=1$.
\end{enumerate}
\end{thm} 
\begin{proof}
We use the Lavers' method to compute the presentation of $TV\BB_n$. Relations (B0) to (B2) and (V1) to (V5) are obtained directly from $V\BB_n$, and relations (P1) to (P3) are obtained directly from $P_n$. Relations (T1) and (T2) are consequences of (P1) and (P2) in $P_n$ because the generators $\mu_{i,j}$ with $i<j-1$ are removed as in Remark \ref{103}. The remaining relations are obtained from Proposition \ref{104} and Proposition \ref{019}, that is, relations (T3), (T4), (TV2), (TV3) and (TV5). Relation (TV4) is redundant and can be obtained from (V1). Indeed, if $i<j<k$ are consecutive numbers in $[n]$, (V1) implies $\mu_{i,k}\nu_i^2=\mu_{i,k}\sigma_i\sigma_i^{-1}$ and $\mu_{i,k}\nu_j^2=\mu_{i,k}\sigma_j\sigma_j^{-1}$. By applying Lemma \ref{101} once on each term, we obtain $\nu_i\eta_j\nu_i=\sigma_i\eta_j\sigma_i^{-1}$ and $\nu_j\eta_i\nu_j=\sigma_j\eta_i\sigma_j^{-1}$, which are equivalent to (TV4). By using (TV4) we obtain (TV1) just as in Remark \ref{022}. Thus, (TV1) is also redundant. (T5) is a consequence of (P3) and can be obtained as in Remark \ref{020}. By \cite[Lemma 2(5--6)]{PreAiJu18} and Remark \ref{020}, relations (P1) to (P3) can be obtained by using (T1), (T2) and (T5) respectively, so they are removed from the presentation of $TV\BB_n$. Finally, relation (TV6) is redundant and obtained from (T5) as in Remark \ref{020}.
\end{proof}

Diagrammatically, (TV1) to (TV6) can be represented, respectively, as follows:\[\figureteneig\]

We will close the section by constructing, according to the Lavers' method, the tied monoid attached to the \emph{virtual singular braid monoid} \cite{CPM16}; this monoid, denoted by $VS\BB_n$, is presented with generators $\sigma_1,\ldots,\sigma_{n-1}$ and $\sigma_1^{-1},\ldots,\sigma_{n-1}^{-1}$ satisfying (B0) to (B2), the generators $\tau_1,\ldots,\tau_{n-1}$ satisfying (S1) to (S3), and the generators $\nu_1,\ldots,\nu_{n-1}$ satisfying (V1) to (V5), all of them subject to the following additional relations:\begin{enumerate}
\item[]{\normalfont(VS1)}\quad$\nu_i\nu_j\tau_i=\tau_j\nu_i\nu_j$ for all $i,j\in[n-1]$ with $|i-j|=1$,
\item[]{\normalfont(VS2)}\quad$\nu_i\tau_j=\tau_j\nu_i$ for all $i,j\in[n-1]$ with $|i-j|\geq2$.
\end{enumerate}
Diagrammatically, (VS1) and (VS2) can be represented as follows:\[\figuretennin\]

There is a natural epimorphism from $VS\BB_n$ to $\SS_n$ mapping $\sigma_i,\tau_i$ and $\nu_i$ to $s_i$ for all $i\in[n-1]$. Hence, we may construct its tied monoid, that is $TVS\BB_n$. We will call $TVS\BB_n$ the \emph{tied virtual singular braid monoid}.

As a consequence of Proposition \ref{021} and Theorem \ref{028}, we obtain the following result.

\begin{thm}\label{017}
The monoid $TVS\BB_n$ can be presented with generators $\sigma_1,\ldots,\sigma_{n-1}$ and $\sigma_1^{-1},\ldots,\sigma_{n-1}^{-1}$ satisfying {\normalfont(B0)} to {\normalfont(B2)}, the generators $\tau_1,\ldots,\tau_{n-1}$ satisfying {\normalfont(S1)} to {\normalfont(S3)}, the generators $\nu_1,\ldots,\nu_{n-1}$ satisfying {\normalfont(V1)} to {\normalfont(V5)} including {\normalfont(VS1)} and {\normalfont(VS2)}, and generators $\eta_1,\ldots,\eta_{n-1}$ satisfying {\normalfont(T1)} to {\normalfont(T5)}, {\normalfont(TS1)} to {\normalfont(TS6)} and {\normalfont(TV1)} to {\normalfont(TV6)}.
\end{thm}

\section{Tied monoids of type B and D }\label{S5}

This section have as goal write out the Lavers's presentation of $T^{\Gamma}M$ for $M=\BB_n^{\Gamma}$, where $\BB_n^\Gamma$ denotes the Artin (braid) monoid of type $\Gamma$ with $\Gamma\in\{B,D\}$. When $\Gamma=B$ (resp. D) the idempotent factor is $P_n^B$ (resp. $P_n^D$). In order to explain the action $\rho$ of $\BB_n^{\Gamma}$ on the idempotent factor, we need to introduce some background. Let $\SS_{\pm n}$ be the group of permutations of $[\pm n]$, and for $k\in[\pm n]$ with $k<n$, we denote by $s_k$ the transposition exchanging $k$ with its immediate successor in $[\pm n]$. We will denote by $\SS_n^B$ the \emph{signed symmetric group} on $2n$ symbols, that is the subgroup of permutations $t\in\SS_{\pm n}$ satisfying $t(-k)=-t(k)$ for all $k\in[\pm n]$, see \cite[Section 8.1]{BjBr05}. Recall that $\SS_n^B$ realizes as the Coxeter group of type B. The Coxeter group of type D can be realized as the \emph{even--signed symmetric group} $\SS_n^D$ on $2n$ symbols, that is the subgroup of signed permutations $t\in\SS_n^B$ satisfying $t(1)\cdots t(n)>0$, see \cite[Section 8.2]{BjBr05}. Observe that we have the following inclusions of subgroups\[\SS_n^D\subset\SS_n^B\subset\SS_{\pm n}.\]Thus, the permutation action of $\SS_{\pm n}$ on $[\pm n]$ induces an action (by restriction) of $\SS_n^B$ on $P_n^B$ and also so an action of $\SS_n^D$ on $P_n^D$, see \ref{135}. These permutation actions will be denoted by $\pp$. So, the action defining $T\BB_n^{\Gamma}$ is the composition of $\pp$ with the natural projection $\pi$ of $\BB_n^{\Gamma}$ in $\SS_n^{\Gamma}$.

To state the Lavers' presentation of $T^B\BB_n^B$, we need  recall that the \emph{Artin (braid) group} $\BB_n^B$, also called \emph{signed braid group}, is presented with generators $\sigma_0,\sigma_1,\ldots,\sigma_{n-1}$ subject to the following relations:\begin{enumerate}
\item[](BB1)\quad$\sigma_0\sigma_1\sigma_0\sigma_1=\sigma_1\sigma_0\sigma_1\sigma_0$,
\item[](BB2)\quad$\sigma_i\sigma_j\sigma_i=\sigma_j\sigma_i\sigma_j$ for all $i,j\in[n-1]$ with $|i-j|=1$,
\item[](BB3)\quad$\sigma_i\sigma_j=\sigma_j\sigma_i$ for all $i,j\in[n-1]$ with $|i-j|\geq2$.
\end{enumerate} 

\begin{thm}\label{031}
The monoid $T^B\BB_n^B$ can be presented with generators $\sigma_0,\sigma_1,\ldots,\sigma_{n-1}$ and $\sigma_0^{-1},\sigma_1^{-1},\ldots,\sigma_{n-1}^{-1}$ satisfying {\normalfont(BB0)} to {\normalfont(BB3)}, and the additional generators $\th_0,\th_1,\ldots,\th_{n-1}$ subject to the relations {\normalfont(TB1)} to {\normalfont(TB9)}, where:
\begin{enumerate}
\item[]{\normalfont(BB0)}\quad$\sigma_i\sigma_i^{-1}=\sigma_i^{-1}\sigma_i=1$ for all $i\in[n-1]\cup\{0\}$,
\item[]{\normalfont(TB1)}\quad$\th_i^2=\th_i$ for all $i\in[n-1]\cup\{0\}$,
\item[]{\normalfont(TB2)}\quad$\th_i\th_j=\th_j\th_i$ for all $i,j\in[n-1]\cup\{0\}$,
\item[]{\normalfont(TB3)}\quad$\th_0\sigma_1\sigma_0\sigma_1=\sigma_1\sigma_0\sigma_1\th_0$ and $\th_1\sigma_0\sigma_1\sigma_0=\sigma_0\sigma_1\sigma_0\th_1$,
\item[]{\normalfont(TB4)}\quad$\sigma_i\sigma_j\th_i=\th_j\sigma_i\sigma_j$  for all $i,j\in[n-1]$ with $|i-j|=1$,
\item[]{\normalfont(TB5)}\quad$\sigma_i\th_j=\th_j\sigma_i$ with $|i-j|\neq1$,
\item[]{\normalfont(TB6)}\quad$\th_i\th_j\sigma_i=\th_j\sigma_i\th_j=\sigma_i\th_i\th_j$ for all $i,j\in[n-1]\cup\{0\}$ with $|i-j|=1$,
\item[]{\normalfont(TB7)}\quad$\th_1\sigma_0\th_1\sigma_0=\sigma_0\th_1\sigma_0\th_1$ and $\th_0\sigma_1\th_0\sigma_1=\sigma_1\th_0\sigma_1\th_0$,
\item[]{\normalfont(TB8)}\quad$\sigma_1^2\th_0=\th_0\sigma_1^2$ and $\sigma_0^2\th_1=\th_1\sigma_0^2$,
\item[]{\normalfont(TB9)}\quad$\sigma_i\sigma_j^{-1}\th_i=\th_j\sigma_i\sigma_j^{-1}$  for all $i,j\in[n-1]$ with $|i-j|=1$.
\end{enumerate}
\end{thm}

Now, in order to state the Lavers's presentation of $T^D\BB_n^D$, we need to recall first that the \emph{Artin (braid) group of type D} or the \emph{even--signed braid group}, which we will denote by $\BB_n^D$, is presented with generators $\sigma_{-1},\sigma_1,\sigma_2,\ldots,\sigma_{n-1}$ subject to the relations:\begin{enumerate}
\item[](DB1)\quad$\sigma_{-1}\sigma_2\sigma_{-1}=\sigma_2\sigma_{-1}\sigma_2$,
\item[](DB2)\quad$\sigma_{-1}\sigma_i=\sigma_i\sigma_{-1}$ for all $i\in[n-1]$ with $i\neq2$,
\item[](DB3)\quad$\sigma_i\sigma_j\sigma_i=\sigma_j\sigma_i\sigma_j$ for all $i,j\in[n-1]$ with $|i-j|=1$,
\item[](DB4)\quad$\sigma_i\sigma_j=\sigma_j\sigma_i$ for all $i,j\in[n-1]$ with $|i-j|\geq2$.
\end{enumerate}

\begin{thm}\label{TMD} 
The monoid $T^D\BB_n^D$ can be presented with generators $\sigma_{-1},\sigma_1,\ldots,\sigma_{n-1}$ and $\sigma_{-1}^{-1},\sigma_1^{-1},\sigma_2\ldots,\sigma_{n-1}^{-1}$ satisfying {\normalfont(DB0)} to {\normalfont(DB4)}, and the additional generators $\th_{-1},\th_1,\th_2,\ldots,\th_{n-1}$ subject to the relations {\normalfont(TD1)} to {\normalfont(TD7)}, where:
\begin{enumerate}
\item[]{\normalfont(DB0)}\quad$\sigma_i\sigma_i^{-1}=\sigma_i^{-1}\sigma_i=1$ for all $i\in[n-1]\cup\{-1\}$,
\item[]{\normalfont(TD1)}\quad$\th_i^2=\th_i$ for all $i\in[n-1]\cup\{-1\}$,
\item[]{\normalfont(TD2)}\quad$\th_i\th_j=\th_j\th_i$ for all $i,j\in[n-1]\cup\{-1\}$,
\item[]{\normalfont(TD3)}\quad$\sigma_i\sigma_j\th_i=\th_j\sigma_i\sigma_j$ with $||i|-|j||=1$,
\item[]{\normalfont(TD4)}\quad$\sigma_i\th_j=\th_j\sigma_i$ with $||i|-|j||\neq1$,
\item[]{\normalfont(TD5)}\quad$\th_i\th_j\sigma_i=\th_j\sigma_i\th_j=\sigma_i\th_i\th_j$ with $||i|-|j||=1$,
\item[]{\normalfont(TD6)}\quad$\th_{-1}\th_1\th_3\sigma_2\sigma_{-1}\sigma_1\sigma_2\th_3=\th_{-1}\th_1\th_2\th_3\sigma_2\sigma_{-1}\sigma_1\sigma_2$,
\item[]{\normalfont(TD7)}\quad$\sigma_i\sigma_j^{-1}\th_i=\th_j\sigma_i\sigma_j^{-1}$ with $||i|-|j||=1$,
\end{enumerate}
\end{thm}
The proof of Theorem \ref{031} and Theorem \ref{TMD} will be done, respectively, at the end of Subsection \ref{proofTMB} and Subsection \ref{proofTMD}.
 
\subsection{Proof of Theorem \ref{031}}\label{proofTMB}

To prove the Theorem \ref {031} we need some technical material, all of which is provided below.
 
For $i\in[n-1]$ denote by $t_i$ the signed transposition exchanging $i$ with $i+1$, and by $t_0$ the transposition exchanging $-1$ with $1$, that is $t_0=s_{-1}$ and $t_i=s_{-(i+1)}s_i$ for all $i\in[n]$. Recall that $\SS_n^B$ is generated by those elements and may be presented with generators $t_0,t_1,\ldots,t_{n-1}$ subject to the following relations:\begin{enumerate}
\item[](BS1)\quad$t_0t_1t_0t_1=t_1t_0t_1t_0$,
\item[](BS2)\quad$t_it_jt_i=t_jt_it_j$ for all $i,j\in[n-1]$ with $|i-j|=1$,
\item[](BS3)\quad$t_it_j=t_jt_i$ for all $i,j\in[n-1]\cup\{0\}$ with $|i-j|\geq2$,
\item[](BS4)\quad$t_i^2=1$ for all $i\in[n-1]\cup\{0\}$.
\end{enumerate} 
  
As the Lavers' method states (Theorem \ref{013}), to construct a presentation for $T^BM$, we need to consider three families of relations: the relations of $M$ and $P_n^B$ which are already defined, and the relations given by the action $\pp\circ\pi$ which we will study below.

Recall that the permutation action of $\SS_n^B$ on $P_n^B$, is given by $t(I):=\{t(I_1),\ldots,t(I_k)\}$ for all $t\in\SS_n^B$ and $I=\{I_1,\ldots,I_k\}\in P_n^B$. Then, for $k\in[n-1]$ and $(i,j)\in[\pm n]\times[n]$ with $|i|<j$, we have
\begin{equation}\label{025}t_k(\ve_{i,j})=\left\{\begin{array}{ll}
\ve_{i-1,j}&\text{if }i\in\{-k,k+1\}\text{ and }i<j,\\
\ve_{i+1,j}&\text{if }i\in\{-k-1,k\}\text{ and }i+1<j,\\
\ve_{i,j-1}&\text{if }j=k+1\text{ and }i+1<j,\\
\ve_{i,j+1}&\text{if }j=k\text{ and }i<j,\\
\ve_{i,j}&\text{if }\{k,k+1\}=\{|i|,j\}\text{ or }\{k,k+1\}\cap\{|i|,j\}=\emptyset.
\end{array}\right.
\end{equation}Furthermore, for $k\in[n-1]$ and $i\in[n]$, we have\begin{equation}\label{133}
t_k(\ve_{-i,i})=\left\{\begin{array}{ll}
\ve_{-(i+1),i+1}&\text{if }i=k,\\
\ve_{-(i-1),i-1}&\text{if }i=k+1,\\
\ve_{-i,i} & \text{if }i\not\in\{k,k+1\}.
\end{array}\right.\end{equation}In the case $k=0$, we have:
\begin{equation}\label{034}
\begin{array}{llll}
t_0(\ve_{i,j})&=&\ve_{-i,j}&\quad\text{for $i=\pm1$, $1<j$,}\\
t_0(\ve_{i,j})&=&\ve_{i,j}&\quad\text{for $i\neq\pm1$, $|i|<j$,}\\
t_0(\ve_{-i,i})&=&\ve_{-i,i}&\quad\text{for $i\in[n]$.}
\end{array} 
\end{equation} Observe that $t_k(\ve_{-j,-i})=t_k(\ve_{i,j})$ for all $i,j\in[n]$ with $i<j$.

We will study the relations of $T^BM$ obtained through the Lavers' method. For that we set $T$ to be the set of pairs  of words $(\ve_{i,j}\sigma_k,\sigma_k\rho_{\sigma_k}(\ve_{i,j}))$ for all $i,j\in[\pm n]$ with $i<j$ and $k\in[n-1]\cup\{0\}$.

\begin{lem}\label{024}
The set $T$ corresponds to the following relations of $T^BM$:\begin{enumerate}
\item $\ve_{-k,j}\sigma_k=\sigma_k\ve_{-(k+1),j}$ and $\ve_{-(k+1),j}\sigma_k=\sigma_k\ve_{-k,j}$ with $k+1<j$,
\item $\ve_{k+1,j}\sigma_k=\sigma_k\ve_{k,j}$ and $\ve_{k,j}\sigma_k=\sigma_k\ve_{k+1,j}$ with $k+1<j$,
\item $\ve_{i,k+1}\sigma_k=\sigma_k\ve_{i,k}$ and $\ve_{i,k}\sigma_k=\sigma_k\ve_{i,k+1}$ with $i<k$,
\item $\ve_{-k,k}\sigma_k=\sigma_k\ve_{-(k+1),k+1}$ and $\ve_{-(k+1),k+1}\sigma_k=\sigma_k\ve_{-k,k}$,
\item $\ve_{-1,j}\sigma_0=\sigma_0\ve_{1,j}$ and $\ve_{1,j}\sigma_0=\sigma_0\ve_{-1,j}$, and $\ve_{i,j}\sigma_0=\sigma_0\ve_{i,j}$ with $1<|i|$,
\item $\ve_{i,j}\sigma_k=\sigma_k\ve_{i,j}$ if $\{k,k+1\}=\{|i|,j\}$ or $\{k,k+1\}\cap\{|i|,j\}=\emptyset$,
\item $\ve_{-i,i}\sigma_k=\sigma_k\ve_{-i,i}$ if $k=0$ or $i\not\in\{k,k+1\}$.
\end{enumerate}
\end{lem}
\begin{proof}
Since $\ve_{i,j}=\ve_{-j,-i}$ for all $i,j\in[\pm n]$ with $i<j$ and $\rho_{\sigma_i}=\pi_{t_i}$ for all $i\in[n-1]\cup\{0\}$, it is obtained by applying \ref{025} to \ref{034}.
\end{proof}

\begin{crl}\label{026}
 The set $T$ corresponds to the following relations of $T^BM$.\begin{enumerate}
\item $\ve_{-(k+1),j}=\sigma_k\ve_{-k,j}\sigma_k^{-1}=\sigma_k^{-1}\ve_{-k,j}\sigma_k$ with $k+1<j$,\label{110}
\item $\ve_{k+1,j}=\sigma_k\ve_{k,j}\sigma_k^{-1}=\sigma_k^{-1}\ve_{k,j}\sigma_k$ with $k+1<j$,\label{111}
\item $\ve_{i,k+1}=\sigma_k\ve_{i,k}\sigma_k^{-1}=\sigma_k^{-1}\ve_{i,k}\sigma_k$ with $|i|<k$,\label{112}
\item $\ve_{-(k+1),k+1}=\sigma_k\ve_{-k,k}\sigma_k^{-1}=\sigma_k^{-1}\ve_{-k,k}\sigma_k$,\label{113}
\item $\ve_{-1,j}=\sigma_0\ve_{1,j}\sigma_0^{-1}=\sigma_0^{-1}\ve_{1,j}\sigma_0$ with $1<j$, and $\ve_{i,j}\sigma_0=\sigma_0\ve_{i,j}$ with $1<|i|$,\label{114}
\item $\ve_{i,j}\sigma_k=\sigma_k\ve_{i,j}$ if $\{k,k+1\}=\{|i|,j\}$ or $\{k,k+1\}\cap\{|i|,j\}=\emptyset$,\label{115}
\item $\ve_{-i,i}\sigma_k=\sigma_k\ve_{-i,i}$ if $k=0$ or $i\not\in\{k,k+1\}$.\label{116}
\end{enumerate}
\end{crl}

Put $\th_0:=\ve_{-1,1}$, and for $i\in[n-1]$ we set $\th_i:=\ve_{i,i+1}$. So, we add these symbols to the presentation of $T^BM$ by using Tietze transformations of type \tf{3}. Note that, by definition, $\th_0,\th_1,\ldots,\th_{n-1}$ satisfy relations (PB1), (PB2) and (PB5) of $P_n^B$.  For each $i\in[n-1]$ we will denote by $\oth_i$ the generator $\ve_{-i,i+1}$.
 
For $i,j\in[n]$ with $i<j$ consider $a_{i,j},\oa_{i,j},b_{i,j}$ and $\ob_{i,j}$ as defined in \ref{029}, and additionally, for each $k\in[n]\backslash\{1\}$, we set $d_k=\sigma_{k-1}\cdots\sigma_1$ and $\od_k=\sigma_1\cdots\sigma_{k-1}$.

\begin{lem}\label{027}
Let $i,j,k\in[n]$ with $i<j$. Then:
\begin{enumerate}
\item $\ve_{i,j}=a_{i,j}\th_{j-1}a_{i,j}^{-1}=\oa_{i,j}^{-1}\th_{j-1}\oa_{i,j}$,\label{052}
\item $\ve_{i,j}=b_{i,j}\th_ib_{i,j}^{-1}=\ob_{i,j}^{-1}\th_i\ob_{i,j}$,\label{151}
\item $\ve_{-i,j}=a_{i,j}\oth_{j-1}$ $a_{i,j}^{-1}=\oa_{i,j}^{-1}\oth_{j-1}$ $\oa_{i,j}$,\label{152}
\item $\ve_{-i,j}=b_{i,j}\oth_ib_{i,j}^{-1}=\ob_{i,j}^{-1}\oth_i\ob_{i,j}$,\label{153}
\item $\ve_{-k,k}=d_k\th_0d_k^{-1}=\od_k^{-1}\th_0\od_k$.
\end{enumerate}
\end{lem}
\begin{proof}
It is obtained by applying Corollary \ref{026} inductively.
\end{proof}

Set $\tb_0=\sigma_0$, and $\tb_k=\sigma_k\tb_{k-1}\sigma_k$ for each $k\in[n-1]$.

\begin{rem}[Theta elements]\label{109}
Each $\tb_k$ is called a \emph{theta element} of the subgroup $\langle\sigma_0,\ldots,\sigma_k\rangle$ of $\BB_n^B$ which were introduced in \cite{ArPa19} to define left orders. The product of these elements, namely $\tb_0\tb_1\cdots\tb_{n-1}$, is the so-called  Garside element of $\BB_n^B$. Because of \ref{025} to \ref{034}, each  $\oth_k$  can be reached by conjugating $\th_k$ by a theta element, that is\begin{equation}\label{128}
\oth_k=\tb_{k-1}\th_k\tb_{k-1}^{-1}=\tb_{k-1}^{-1}\th_k\tb_{k-1},\qquad k\in[n-1].\end{equation}Furthermore, for every $i,j\in[n-1]$ with $i<j$, we have the following relations for the elements $a_{i,j}$ defined in \ref{029}\begin{equation}\label{117}a_{i,j}\tb_{j-2}^{-1}=\tb_{i-1}^{-1}\oa_{i,j}^{-1},\qquad\oa_{i,j}^{-1}\tb_{j-2}=\tb_{i-1}a_{i,j}.\end{equation}By \cite[Lemma 3.3(1)]{ArPa19}, we have $\tb_kg=g\tb_k$ for all $g\in\langle\sigma_0,\ldots,\sigma_{k-1},\sigma_{k+2},\ldots,\sigma_{n-1}\rangle$.
\end{rem}

\begin{crl}\label{030}
The set $T$ corresponds to the following relations in $T^BM$:
\begin{enumerate}
\item$\ve_{i,j}=a_{i,j}\th_{j-1}a_{i,j}^{-1}=\ob_{i,j}^{-1}\th_i\ob_{i,j}=\oa_{i,j}^{-1}\th_{j-1}\oa_{i,j}=b_{i,j}\th_ib_{i,j}^{-1}$ with $1\leq i<j$,\label{047}
\item$\ve_{-i,j}=a_{i,j}(\tb_{j-2}^{-1}\th_{j-1}\tb_{j-2})a_{i,j}^{-1}=\ob_{i,j}^{-1}(\tb_{i-1}\th_i\tb_{i-1}^{-1})\ob_{i,j}=\oa_{i,j}^{-1}(\tb_{j-2}\th_{j-1}\tb_{j-2}^{-1})\oa_{i,j}\\=b_{i,j}(\tb_{i-1}^{-1}\th_i\tb_{i-1})b_{i,j}^{-1}$ with $1\leq i<j$,\label{048}
\item$\ve_{-k,k}=d_k\th_0d_k^{-1}=\od_k^{-1}\th_0\od_k$ with $1\leq k$,\label{049}
\item$(\oa_{i,j}^{-1}\th_{j-1}\oa_{i,j})\sigma_k=\sigma_k(a_{i,j}\th_{j-1}a_{i,j}^{-1})$ if $(k,k+1)=(i,j)$ or $\{k,k+1\}\cap\{i,j\}=\emptyset$ with $k\geq0$ and $1\leq i<j$,\label{051}
\item$(\oa_{i,j}^{-1}\tb_{j-2}\th_{j-1}\tb_{j-2}^{-1}\oa_{i,j})\sigma_k=\sigma_k(a_{i,j}\tb_{j-2}^{-1}\th_{j-1}\tb_{j-2}a_{i,j}^{-1})$ if $(k,k+1)=(i,j)$ or $\{k,k+1\}\cap\{i,j\}=\emptyset$ with $k\geq0$ and $1\leq i<j$,\label{136}
\item$(\od_i^{-1}\th_0\od_i)\sigma_k=\sigma_k(d_i\th_0d_i^{-1})$ if $k=0$ or $k\geq1$ and $i\not\in\{k,k+1\}$.\label{053}
\end{enumerate}
\end{crl}
\begin{proof}
It is a direct consequence of Lemma \ref{027}, Corollary \ref{026}  and Equation \ref{128}.
\end{proof}

\begin{rem}\label{058}
Lemma \ref{027} implies that each generator $\ve_{i,j}$ can be replaced by expressions with $\th$'s in all relations of $T^BM$, in particular in relations (PB1) to (PB5). Hence, as in Corollary \ref{030}, each $\ve_{i,j}$ appears at most once in the relations of $T^BM$. So, Proposition \ref{012} implies that we can remove every $\ve_{i,j}$ from the presentation of $T^BM$. Note that relations (PB1) to (PB5) were not removed from the presentation, however they appear with $\ve_{i,j}$ replaced by their expressions given in Lemma \ref{027}.  Moreover, via conjugation, every $\ve_{i,j}$ may be reached by $\th_1$ or $\th_0$, indeed, for $i,j,k\in[n]$ with $i<j$, we have:\[\ve_{i,j}=d_ib_{1,j}\th_1b_{1,j}^{-1}d_i^{-1},\quad\ve_{-i,j}=d_ib_{1,j}\sigma_0\th_1\sigma_0^{-1}b_{1,j}^{-1}d_i^{-1}\quad\text{and}\quad\ve_{-k,k}=d_k\th_0d_k^{-1}.\]
\end{rem}

Consider the following relations in the free monoid  $\left(\{\sigma_1,\ldots,\sigma_{n-1}\}\cup\{\th_0,\ldots,\th_{n-1}\}\right)^*$.\begin{align}
\sigma_i\sigma_j\th_i&=\th_j\sigma_i\sigma_j,&\mkern-54mu|i-j|&=1,\quad i,j\geq1,\label{137}\\
\sigma_i\sigma_j^{-1}\th_i&=\th_j\sigma_i\sigma_j^{-1},&\mkern-54mu|i-j|&=1,\quad i,j\geq1,\label{142}\\
\sigma_i\th_j&=\th_j\sigma_i,&\mkern-54mu|i-j|&\neq1,\label{139}\\
\sigma_1^2\th_0&=\th_0\sigma_1^2,\label{138}\\
\sigma_0^2\th_1&=\th_1\sigma_0^2,\label{143}\\
\th_0\sigma_1\sigma_0\sigma_1&=\sigma_1\sigma_0\sigma_1\th_0,\label{140}\\
\th_1\sigma_0\sigma_1\sigma_0&=\sigma_0\sigma_1\sigma_0\th_1\label{144}.
\end{align}
 
\begin{lem}\label{145}
Relations \ref{137} to \ref{144} hold in $T^BM$ and can be deduced from $T$.
\end{lem}
\begin{proof}
Relations \ref{137} and \ref{142} are obtained by taking $j=i+2$ in Corollary \ref{030}(\ref{047}). We get \ref{139} by taking $(k,i)=(0,1)$ in Corollary \ref{030}(\ref{053}) and $j=i+1$ in Corollary \ref{030}(\ref{051}). Relation \ref{138} is obtained from Corollary \ref{030}(\ref{049}) with $k=2$. We get relation \ref{143} by taking $(i,j)=(1,2)$ in the second equality of Corollary \ref{030}(\ref{048}). We obtain relation \ref{140} by taking by taking $(k,i)=(0,2)$ in Corollary \ref{030}(\ref{053}). By using \ref{143} and taking $(i,j,k)=(1,2,1)$ in Corollary \ref{030}(\ref{136}) we get relation \ref{144}.
\end{proof}

Relation (BS2) together with \ref{142} and \ref{139} imply the following relations:
\begin{align}
\th_i\sigma_j^{-1}\sigma_i&=\sigma_j^{-1}\sigma_i\th_j,&|i-j|&=1,\quad i,j\geq1,\label{146}\\
\sigma_i^2\th_j&=\th_j\sigma_i^2,&|i-j|&=1,\quad i,j\geq1.\label{141}
\end{align}

\begin{pro}\label{057}
The congruence $\overline{T}$ is generated by the relations \ref{137} to \ref{144}.
\end{pro}
\begin{proof}
Because of Lemma \ref{145}, it is enough to show that the relations in Corollary \ref{030} can be deduced from \ref{137} to \ref{144}. Corollary \ref{030}(\ref{047}) is equivalent to:\[\begin{array}{rcccl}
\ob_{i,j}a_{i,j}\th_{j-1}&=&(\sigma_{i+1}\sigma_i)\cdots(\sigma_{j-1}\sigma_{j-2})\th_{j-1}&=&\th_i\ob_{i,j}a_{i,j},\\
\oa_{i,j}\ob_{i,j}^{-1}\th_i&=&(\sigma_{j-2}\sigma_{j-1}^{-1})\cdots(\sigma_i\sigma_{i+1}^{-1})\th_i&=&\th_{j-1}\oa_{i,j}\ob_{i,j}^{-1},\\
\oa_{i,j}b_{i,j}\th_i&=&(\sigma_{j-2}\sigma_{j-1})\cdots(\sigma_i\sigma_{i+1})\th_i&=&\th_{j-1}\oa_{i,j}b_{i,j}.
\end{array}\]which can be deduced from \ref{137} and \ref{142}. We will study now Corollary \ref{030}(\ref{048}). Equation \ref{117} and the fact that $b_{i,j}\tb_{i-1}=\tb_{i-1}b_{i,j}$ and $\ob_{i,j}\tb_{i-1}=\tb_{i-1}\ob_{i,j}$ for all $i,j\in[n]$ with $i<j$ imply that the relations obtained from Corollary \ref{030}(\ref{048}) are:\begin{equation}\label{054}\tb_{i-1}^2\ob_{i,j}^{-1}\th_i\ob_{i,j}=\oa_{i,j}^{-1}\th_{j-1}\oa_{i,j}\tb_{i-1}^2,\quad\tb_{i-1}^2a_{i,j}\th_{j-1}a_{i,j}^{-1}=b_{i,j}\th_ib_{i,j}^{-1}\tb_{i-1}^2.\end{equation}In particular, for $j=i+1$ we obtain $\tb_{i-1}^2\th_i=\th_i\tb_{i-1}^2$. By applying this relations together with \ref{137} in \ref{054} we get the following additional relations that can be deduced from \ref{146}.\[a_{i,j}^{-1}b_{i,j}\th_i=(\sigma_{j-2}^{-1}\sigma_{j-1})\cdots(\sigma_i^{-1}\sigma_{i+1})\th_i=\th_{j-1}a_{i,j}^{-1}b_{i,j}.\] 

To finish with Corollary \ref{030}(\ref{048}) we need to show that $\tb_{i-1}^2\th_i=\th_i\tb_{i-1}^2$ can be deduced from \ref{137} to \ref{144}. Observe that \cite[Lemma 3.3(1)]{ArPa19} implies $\tb_{i-1}\sigma_i\tb_{i-1}\sigma_i=\sigma_i\tb_{i-1}\sigma_i\tb_{i-1}$. So, each $\tb_i^2$ can be decomposed in $M$ as follows\[\tb_i^2=\sigma_i\tb_{i-1}\sigma_i\tb_{i-1}^{-1}\tb_{i-1}\sigma_i\tb_{i-1}\sigma_i=\tb_{i-1}^{-1}\sigma_i\sigma_{i+1}^{-1}\tb_{i-1}^2\sigma_{i+1}\sigma_i\tb_{i-1}\sigma_i^2.\]

A simple induction with \ref{143} as hypothesis together with \ref{137} and \ref{142} imply that $\tb_{i-1}^2\th_i=\th_i\tb_{i-1}^2$. Now, Corollary \ref{030}(\ref{049}) can be rewritten as $\od_kd_k\th_0=\th_0\od_kd_k$ which is deduced from the fact that $\od_kd_k=\sigma_2^{-1}\cdots\sigma_k^{-1}\od_{k-1}d_{k-1}\sigma_k\cdots \sigma_2\sigma_1^2$ by using \ref{139} and \ref{138}. We study now Corollary \ref{030}(\ref{051}); this relation can be rewritten as $\th_{j-1}\oa_{i,j}\sigma_ka_{i,j}=\oa_{i,j}\sigma_ka_{i,j}\th_{j-1}$ which can be also deduced from \ref{137} to \ref{139} together with \ref{141}. Indeed, we have $\th_{j-1}\oa_{i,j}a_{i,j}=\oa_{i,j}a_{i,j}\th_{j-1}$ because\[\oa_{i,j}a_{i,j}=\sigma_{j-2}\sigma_{j-1}\cdots\sigma_{i+1}\sigma_{i+2}\sigma_i^2\sigma_{i+2}^{-1}\sigma_{i+1}\cdots\sigma_{j-1}^{-1}\sigma_{j-2}.\]If $k<i-1$ or $k>j$ we are done because $\sigma_k\oa_{i,j}=\oa_{i,j}\sigma_k$ and $\sigma_ka_{i,j}=a_{i,j}\sigma_k$. Otherwise,\[\th_{j-1}\oa_{i,j}\sigma_ka_{i,j}=\th_{j-1}\oa_{i,j}a_{i,j}\sigma_{k-1}=\oa_{i,j}a_{i,j}\th_{j-1}\sigma_{k-1}=\oa_{i,j}a_{i,j}\sigma_{k-1}\th_{j-1}=\oa_{i,j}\sigma_ka_{i,j}\th_{j-1}.\]By using \ref{117} the relation Corollary \ref{030}(\ref{136}) corresponds to:\begin{align}\label{007}\oa_{i,j}\tb_{i-1}\sigma_k^{-1}\tb_{i-1}a_{i,j}\th_{j-1}&=\th_{j-1}\oa_{i,j}\tb_{i-1}\sigma_k^{-1}\tb_{i-1}a_{i,j}&j>i+1,&\\
\tb_{i-1}\sigma_k^{-1}\tb_{i-1}\th_i&=\th_i\tb_{i-1}\sigma_k^{-1}\tb_{i-1}&j=i+1.\label{147}\end{align}Note that \ref{147} and the fact that $a_{i,j}=a_{i+1,j+1}^{-1}\sigma_{i+1}\sigma_i\cdots\sigma_{j+1}\sigma_j$ imply \ref{007}. So, it is enough to deduce \ref{147}. If $\{i,j\}\cap\{k,k+1\}=\emptyset$ it is obtained because of \ref{054} and \cite[Lemma 3.3(1)]{ArPa19}. If $(i,j)=(k,k+1)$, \ref{054} with $j=i+1$ implies that \ref{147} can be written as $\tb_{i-1}\sigma_i\tb_{i-1}\th_i=\th_i\tb_{i-1}\sigma_i\tb_{i-1}$ which is deduced from \ref{137}  and \ref{143} as\[\tb_{i-1}\sigma_i\tb_{i-1}\th_i=\sigma_{i-1}\sigma_i\cdots\sigma_1\sigma_2\sigma_0^2\sigma_2\sigma_1\cdots \sigma_i\sigma_{i-1}\th_i=\th_i\tb_{i-1}\sigma_i\th_{i-1}.\]Finally, we study Corollary \ref{030}(\ref{053}). If $k=0$ we have\[\od_i\sigma_0d_i=\sigma_1\sigma_0\sigma_1\sigma_2\cdots\sigma_i\od_{i-1}d_{i-1}\sigma_i^{-1}\cdots\sigma_2^{-1}.\]If $i<k$ then $\od_i\sigma_kd_i=\sigma_k\od_id_i$, and if $0<k<i-1$ we have $\od_i\sigma_kd_i=\sigma_{k+1}\od_id_i$. So, by using \ref{138} and \ref{140} we obtain the relations. Thus, the proof is concluded.
\end{proof}

Now, as we remove generators $\ve$'s and add $\th$'s (Remark \ref{058}), we need to rewrite relations of $P_n^B$ in terms of these generators. Lemma \ref{027} implies that (PB1) can be deduced from the involution of $\th$'s. Note that (PB5) becomes trivial. Relations (PB2), (PB3) and (PB4) are studied, respectively, in Lemma \ref{129}, Lemma \ref{154} and Remark \ref{059}.

\begin{lem}\label{129}
Relations in \normalfont{(PB2)} can be obtained from commutativity of $\th$'s and the relation $\ve_{i,j}\ve_{r,s}=\ve_{r,s}\ve_{i,j}$ with $(i,j,r,s)\in\{(-1,2,1,2),(-1,1,-2,2)\}$.
\end{lem}
\begin{proof}
Denote by (PB2)$(i,j,r,s)$ the relation $\ve_{i,j}\ve_{r,s}=\ve_{r,s}\ve_{i,j}$ such that $i,j,r,s\in[\pm n]$ with $i<j$ and $r<s$. Note that (PB2)$(i,j,r,s)$ is equivalent to (PB2)$(r,s,i,j)$. Further, since $\ve_{-j,-i}=\ve_{i,j}$ for all $i,j\in[\pm n]$ with $i<j$, then (PB2)$(i,j,r,s)$ is equivalent to (PB2)$(i,j,-s,-r)$, (PB2)$(-j,-i,r,s)$ and (PB2)$(-j,-i,-s,-r)$. So, we may assume that $j,s\geq1$. Let $i,j,r,s\in[n]$ such that $i<j$ and $r<s$. The relation (PB2)$(i,j,r,s)$ can be deduced from commutativity of $\th$'s as in Remark \ref{020}. The relation (PB2)$(-i,j,r,s)$ is equivalent to the following:\[\left\{\begin{array}{ll}
d_ib_{1,j}\text{(PB2)}(-1,2,1,2)(d_ib_{1,j})^{-1}&\text{if }r=i,\,\,s=j,\\
d_ib_{1,j}\sigma_1\sigma_0\text{(PB2)}(1,2,2,s)(d_ib_{1,j}\sigma_1\sigma_0)^{-1}&\text{if }r=i,\,\,s>j,\\
d_ib_{1,j}\sigma_1\sigma_0\text{(PB2)}(1,2,2,s+1)(d_ib_{1,j}\sigma_1\sigma_0)^{-1}&\text{if }r=i,\,\,s<j,\\
d_i\sigma_0\text{(PB2)}(1,j,r,r+1)(d_i\sigma_0)^{-1}&\text{if }r<i=s,\\
d_i\sigma_0\text{(PB2)}(1,j,r+1,s+1)(d_i\sigma_0)^{-1}&\text{if }r<i>s,\\
d_i\sigma_0\text{(PB2)}(1,j,r+1,s)(d_i\sigma_0)^{-1}&\text{if }r<i<s,\\
d_i\sigma_0\text{(PB2)}(1,j,r,s)(d_i\sigma_0)^{-1}&\text{if }r>i.
\end{array}\right.\]Similarly, the relation (PB2)$(-i,j,-r,s)$ can be written as one of the relations above\[\left\{\begin{array}{ll}
d_i\sigma_0\text{(PB2)}(1,j,1,s)(d_i\sigma_0)^{-1}&\text{if }r=i,\\
d_i\sigma_0\text{(PB2)}(1,j,-r,s)(d_i\sigma_0)^{-1}&\text{if }r>i,\\
d_r\sigma_0\text{(PB2)}(-i,j,1,s)(d_i\sigma_0)^{-1}&\text{if }r<i.
\end{array}\right.\]Relation (PB2)$(-i,i,\pm r,s)$ can be deduced from the commutativity of $\th$'s, because\[\left\{\begin{array}{ll}
d_ib_{1,r}b_{2,s}\text{(PB2)}(-1,1,\pm2,3)(d_ib_{1,r}b_{2,s})^{-1}&\text{if }i<r<s,\\
d_ib_{1,r+1}b_{2,s}\text{(PB2)}(-1,1,\pm2,3)(d_ib_{1,r+1}b_{2,s})^{-1}&\text{if }r<i<s,\\
d_ib_{1,r+1}b_{2,s+1}\text{(PB2)}(-1,1,\pm2,3)(d_ib_{1,r+1}b_{2,s+1})^{-1}&\text{if }r<s<i,\\
d_ib_{1,s}\text{(PB2)}(-1,1,\pm1,2)(d_ib_{1,s})^{-1}&\text{if }i=r<s,\\
d_ib_{1,r}b_{2,r+1}\text{(PB2)}(-1,1,\pm2,3)(d_ib_{1,r}b_{2,r+1})^{-1}&\text{if }r<s=i.\end{array}\right.\]where\[\begin{array}{rcl}\text{(PB2)}(-1,1,-2,3)=\tb_1\text{(PB2)}(-1,1,2,3)\tb_1^{-1},\\
\text{(PB2)}(-1,1,-1,2)=\tb_0\text{(PB2)}(-1,1,1,2)\tb_0^{-1}.\end{array}\]Finally, we have (PB2)$(-i,i,-j,j)=d_ib_{1,j}\text{(PB2)}(-1,1,-2,2)(d_ib_{1,j})^{-1}$.
\end{proof}

\begin{lem}\label{154}
Relations in {\normalfont(PB3)} can be obtained from $\ve_{i,j}\ve_{i,k}=\ve_{i,j}\ve_{j,k}=\ve_{i,k}\ve_{j,k}$ such that $(i,j,k)\in\{(1,2,3), (-1,1,2),(-2,1,2)\}$.
\end{lem}
\begin{proof}
Denote by (PB3)$(i,j,k)$ the relation $\ve_{i,j}\ve_{i,k}=\ve_{i,j}\ve_{j,k}=\ve_{i,k}\ve_{j,k}$ with $i,j,k\in[\pm n]$ and $i<j<k$. Since $\ve_{-j,-i}=\ve_{i,j}$ for all $i,j\in[\pm n]$ with $i<j$, then (PB3)$(i,j,k)$ is equivalent to (PB3)$(-k,-j,-i)$, so, we may assume that $j,k\geq1$. Let $i,j,k\in[n]$ with $i<j<k$. The relation (PB3)$(i,j,k)$ can be obtained from (PB3)$(1,2,3)$ as in Remark \ref{020}. Relation (PB3)$(-i,j,k)$ is deduced from (PB3)$(1,2,3)$ as well, because\[\text{(PB3)}(-i,j,k)=\tb_{i-1}\text{(PB3)}(i,j,k)\tb_{i-1}^{-1}.\] 
Similarly, (PB3)$(-j,i,k)$ and (PB3)$(-k,i,j)$ can be deduced  from (PB3)$(1,2,3)$, because \[\begin{array}{rcl}
\text{(PB3)}(-j,i,k)&=&d_ib_{1,j}b_{2,k}\sigma_1\sigma_0\text{(PB3)}(1,2,3)(d_ib_{1,j}b_{2,k}\sigma_1\sigma_0)^{-1},\\
\text{(PB3)}(-k,i,j)&=&d_ib_{1,j}b_{1,k}\sigma_1\sigma_0\text{(PB3)}(1,2,3)(d_ib_{1,j}b_{1,k}\sigma_1\sigma_0)^{-1}.
\end{array}\]Finally, (PB3)$(-i,i,j)$ and (PB3)$(-j,i,j)$ are obtained from (PB3)$(-1,1,2)$ and (PB3)$(-2,1,2)$ respectively, because\[\begin{array}{rcl}
\text{(PB3)}(-i,i,j)&=&d_ib_{1,j}\text{(PB3)}(-1,1,2)(d_ib_{1,j})^{-1},\\
\text{(PB3)}(-j,i,j)&=&d_ib_{1,j}\text{(PB3)}(-2,1,2)(d_ib_{1,j})^{-1}.
\end{array}\]This concludes the proof.
\end{proof}

\begin{rem}\label{059}
 Lemma \ref{129} implies that (PB2) in $T^BM$ is obtained from the commutativity of $\th$'s together with the following relations:\[\ve_{-1,2}\ve_{1,2}=\ve_{1,2}\ve_{-1,2},\qquad\ve_{-1,1}\ve_{-2,2}=\ve_{-2,2}\ve_{-1,1}.\]Because of Lemma \ref{027}, these relations can be rewritten as follows:\[\th_1\sigma_0\th_1\sigma_0=\sigma_0\th_1\sigma_0\th_1,\qquad\th_0\sigma_1\th_0\sigma_1=\sigma_1\th_0\sigma_1\th_0.\]On the other hand, relation {\normalfont(PB4)} in $T^BM$ can be obtained from $\ve_{-1,2}\ve_{1,2}=\ve_{-1,1}\ve_{-2,2}$. Indeed, for $i,j\in[n]$ with $i<j$, we have\[\ve_{-i,j}\ve_{i,j}=d_ib_{1,j}\ve_{-1,2}\ve_{1,2}b_{1,j}^{-1}d_i^{-1},\quad\ve_{-i,i}\ve_{-j,j}=d_ib_{1,j}\ve_{-1,1}\ve_{-2,2}b_{1,j}^{-1}d_i^{-1}.\]This together with Lemma \ref{154} imply that, for $i<j<k$ consecutive in $[n]$, (PB3) and (PB4) can be replaced by the following relations:\[\begin{array}{c}\ve_{i,j}\ve_{i,k}=\ve_{i,j}\ve_{j,k}=\ve_{i,k}\ve_{j,k},\\
\ve_{-1,1}\ve_{1,2}=\ve_{-1,1}\ve_{-1,2}=\ve_{-1,2}\ve_{1,2}=\ve_{-1,2}\ve_{-2,2}=\ve_{1,2}\ve_{-2,2}=\ve_{-1,1}\ve_{-2,2}.\end{array}\]So, as in Remark \ref{020}, we obtain the equivalent relations:\[\begin{array}{c}\th_i\th_j\sigma_i=\th_j\sigma_i\th_j=\sigma_i\th_i\th_j,\quad{ i,j\in[n-1]\cup\{0\},\,\,\,|i-j|=1.}
\end{array}\]
\end{rem}

\begin{proof}[Proof of Theorem \ref{031}]
Relations (BB0) to (BB3) and (PB1) to (PB5) are obtained from $\BB_n^B$ and $P_n^B$ respectively. Relations (TB3), (TB4), (TB5), (TB8) and (TB9) are obtained from Proposition \ref{057} and so $\ve$'s and (PB5) are removed as in Remark \ref{058}. We add (TB1) and (TB2), so remove (PB1) and (PB2) because they are deduced from the first ones. We finally add redundantly relations (TB6) and (TB7), which are constructed and equivalent to (PB2), (PB3) and (PB4) as in Remark \ref{059}.
\end{proof}

\subsection{Proof of Theorem \ref{TMD}}\label{proofTMD}

Below we give the necessary  ingredients to prove Theorem \ref{TMD}. This proof is done at the end of the present section.

Let $t_{-1}$ be the signed transposition exchanging $-1$ with $2$, that is $t_{-1}=(-2\,\,1)(-1\,\,2)$. The subgroup $\SS_n^D$ is generated by $t_{-1},t_1,\ldots,t_{n-1}$ and may be presented with these generators subject to the following relations: 
\begin{enumerate}
\item[](DS1)\quad$t_{-1}t_2t_{-1}=t_2t_{-1}t_2$,
\item[](DS2)\quad$t_{-1}t_i=t_it_{-1}$ for all $i\in[n-1]$ with $i\neq2$,
\item[](DS3)\quad$t_it_jt_i=t_jt_it_j$ for all $i,j\in[n-1]$ with $|i-j|=1$,
\item[](DS4)\quad$t_it_j=t_jt_i$ for all $i,j\in[n-1]$ with $|i-j|\geq2$,
\item[](DS5)\quad$t_i^2=1$ for all $i\in[n-1]\cup\{-1\}$.
\end{enumerate}

We describe now the permutation action of the generators $t_k$'s on the $\ve_{i,j}$'s. More precisely, for $(i,j)\in[\pm n]\times[n]$ with $|i|<j$, we have \ref{025} and the following equation\begin{equation}\label{108}t_{-1}(\ve_{i,j})=\left\{\begin{array}{ll}
\ve_{-2i,j}&\text{if }i=\pm1\text{ and }j>2,\\
\ve_{-\frac{i}{2},j}&\text{if }i=\pm2\text{ and }j>2,\\[-0.5mm] 
\ve_{i,j}&\text{if }(i,j)=(\pm1,2)\text{ or }\{|i|,j\}\cap\{1,2\}=\emptyset.
\end{array}\right.\end{equation} Observe that $t_k(\ve_{-j,-i})=t_k(\ve_{i,j})$ for all $i,j\in[n]$ with $i<j$ and $j\neq-i$.
 
We will study the relations of $T^DM$ obtained through the Lavers' method. For that we set $V$ to be the set of pairs  of words $(\ve_{i,j}\sigma_k,\sigma_k\rho_{\sigma_k}(\ve_{i,j}))$ for all $i,j\in[\pm n]$ with $|i|<j$ and $k\in[n-1]\cup\{-1\}$.

\begin{lem}
The set $V$ corresponds to the following relations of $T^DM$:\begin{enumerate}
\item$\ve_{-k,j}\sigma_k=\sigma_k\ve_{-(k+1),j}$ and $\ve_{-(k+1),j}\sigma_k=\sigma_k\ve_{-k,j}$,
\item$\ve_{k+1,j}\sigma_k=\sigma_k\ve_{k,j}$ and $\ve_{k,j}\sigma_k=\sigma_k\ve_{k+1,j}$ with $k+1<j$,
\item$\ve_{i,k+1}\sigma_k=\sigma_k\ve_{i,k}$ and $\ve_{i,k}\sigma_k=\sigma_k\ve_{i,k+1}$ with $i<k$,
\item$\ve_{-1,j}\sigma_{-1}=\sigma_{-1}\ve_{2,j}$ and $\ve_{1,j}\sigma_{-1}=\sigma_{-1}\ve_{-2,j}$ with $j>2$,
\item$\ve_{-2,j}\sigma_{-1}=\sigma_{-1}\ve_{1,j}$ and $\ve_{2,j}\sigma_{-1}=\sigma_{-1}\ve_{-1,j}$ with $j>2$,
\item$\ve_{i,j}\sigma_k=\sigma_k\ve_{i,j}$ if $\{|k|,|k|+1\}=\{|i|,j\}$ or $\{|k|,|k|+1\}\cap\{|i|,j\}=\emptyset$.
\end{enumerate}
\end{lem}
\begin{proof}
Parts (1) to (3) and some relations of (6) are obtained from Lemma \ref{024}. The rest of relations are obtained directly from \ref{108}.
\end{proof}

\begin{crl}\label{119}
The set $V$ corresponds to the following relations of $T^DM$:
\begin{enumerate}
\item$\ve_{-(k+1),j}=\sigma_k^{-1}\ve_{-k,j}\sigma_k=\sigma_k\ve_{-k,j}\sigma_k^{-1}$ with $k+1<j$,
\item$\ve_{k+1,j}=\sigma_k\ve_{k,j}\sigma_k^{-1}=\sigma_k^{-1}\ve_{k,j}\sigma_k$ with $k+1<j$,
\item$\ve_{i,k+1}=\sigma_k\ve_{i,k}\sigma_k^{-1}=\sigma_k^{-1}\ve_{i,k}\sigma_k$ with $|i|<k$,
\item$\ve_{-1,j}=\sigma_{-1}\ve_{2,j}\sigma_{-1}^{-1}=\sigma_{-1}^{-1}\ve_{2,j}\sigma_{-1}$ with $j>2$,
\item$\ve_{-2,j}=\sigma_{-1}\ve_{1,j}\sigma_{-1}^{-1}=\sigma_{-1}^{-1}\ve_{1,j}\sigma_{-1}$ with $j>2$,
\item$\ve_{i,j}\sigma_k=\sigma_k\ve_{i,j}$ if $\{|k|,|k|+1\}=\{|i|,j\}$ or $\{|k|,|k|+1\}\cap\{|i|,j\}=\emptyset$.\label{161}
\end{enumerate}
\end{crl}

Let $\th_{-1}:=\ve_{-1,2}$. We add $\th_{-1},\th_1,\ldots,\th_{n-1}$ to the presentation of $T^DM$ by using Tietze transformations of type \tf{3}. Note that, by definition, these symbols satisfy relations (PD1), (PD2) and (PD4) of $P_n^D$. In particular, we have $\oth_1=\th_{-1}$.

\begin{lem}\label{120}
Let $i,j,k\in[n]$ with $i<j$. Then
\begin{enumerate}
\item$\ve_{i,j}=a_{i,j}\th_{j-1}a_{i,j}^{-1}=\oa_{i,j}^{-1}\th_{j-1}\oa_{i,j}$,\label{125}
\item$\ve_{i,j}=b_{i,j}\th_ib_{i,j}^{-1}=\ob_{i,j}^{-1}\th_i\ob_{i,j}$,\label{156}
\item$\ve_{-i,j}=a_{i,j}\oth_{j-1}a_{i,j}^{-1}=\oa_{i,j}^{-1}\oth_{j-1}\oa_{i,j}$,\label{157}
\item$\ve_{-i,j}=b_{i,j}\oth_ib_{i,j}^{-1}=\ob_{i,j}^{-1}\oth_i\ob_{i,j}$,\label{158}
\item$\ve_{-1,j}=\sigma_{-1}b_{2,j}\th_2b_{2,j}^{-1}\sigma_{-1}^{-1}=\sigma_{-1}^{-1}\oa_{2,j}^{-1}\th_{j-1}\oa_{2,j}\sigma_{-1}$,\label{159}
\item$\ve_{-2,j}=\sigma_{-1}b_{1,j}\th_1b_{1,j}^{-1}\sigma_{-1}^{-1}=\sigma_{-1}^{-1}\oa_{1,j}^{-1}\th_{j-1}\oa_{1,j}\sigma_{-1}$.\label{160}
\end{enumerate}
\end{lem}
\begin{proof}
It is obtained by applying Corollary \ref{119} inductively.
\end{proof}

Let $\td_1=\sigma_{-1}\sigma_1$ and $\td_i:=\sigma_i\td_{i-1}\sigma_i$  for each $k\in[n-1]$ with $k\geq2$. As in Remark \ref{109}, the $\td_i$'s are also called \emph{theta elements} of $\BB_n^D$, which satisfy \ref{128} for $k\geq2$ and \ref{117}  for $j>i\geq2$ respectively.

\begin{crl}\label{118} 
The set $V$ corresponds to the following relations of $T^DM$. 
\begin{enumerate}
\item$\ve_{i,j}=a_{i,j}\th_{j-1}a_{i,j}^{-1}=\ob_{i,j}^{-1}\th_i\ob_{i,j}=\oa_{i,j}^{-1}\th_{j-1}\oa_{i,j}=b_{i,j}\th_ib_{i,j}^{-1}$ with $1\leq i<j$,\label{121}
\item$\ve_{-i,j}=a_{i,j}(\td_{j-2}^{-1}\th_{j-1}\td_{j-2})a_{i,j}^{-1}=\ob_{i,j}^{-1}(\td_{i-1}\th_i\td_{i-1}^{-1})\ob_{i,j}=\oa_{i,j}^{-1}(\td_{j-2}\th_{j-1}\td_{j-2}^{-1})\oa_{i,j}=b_{i,j}(\td_{i-1}^{-1}\th_i\td_{i-1})b_{i,j}^{-1}$ with $2\leq i<j$,\label{122}
\item$\ve_{-1,j}=a_{1,j}(\td_{j-2}^{-1}\th_{j-1}\td_{j-2})a_{1,j}^{-1}=\ob_{1,j}^{-1}\th_{-1}\ob_{1,j}=\oa_{1,j}^{-1}(\td_{j-2}\th_{j-1}\td_{j-2}^{-1})\oa_{1,j}=b_{1,j}\th_{-1}b_{1,j}^{-1}$  with $j\geq3$,\label{165}
\item$(\oa_{i,j}^{-1}\th_{j-1}\oa_{i,j})\sigma_k=\sigma_k(a_{i,j}\th_{j-1}a_{i,j}^{-1})$ if $(|k|,|k|+1)=(i,j)$ or \mbox{$\{|k|,|k|+1\}\cap\{i,j\}=\emptyset$}.\label{127}
\item$(\oa_{i,j}^{-1}\td_{j-2}\th_{j-1}\td_{j-2}^{-1}\oa_{i,j})\sigma_k=\sigma_k(a_{i,j}\td_{j-2}^{-1}\th_{j-1}\td_{j-2}a_{i,j}^{-1})$ if $(|k|,|k|+1)=(i,j)$ or $\{|k|,|k|+1\}\cap\{i,j\}=\emptyset$.\label{163}
\end{enumerate}
\end{crl}
\begin{proof}
It is a consequence of Lemma \ref{120} and Corollary \ref{119} (\ref{161}) together with the fact that relations in Lemma \ref{120}(\ref{160}) is part of (\ref{122}).
\end{proof}
 
As in Remark \ref{058}, each generator $\ve_{i,j}$ can be replaced by expressions with $\th$'s in all relations of $T^DM$. Hence, as in Corollary \ref{118}, each $\ve_{i,j}$ appears at most once and so Proposition \ref{012} implies that we can remove every $\ve_{i,j}$ from the presentation of $T^DM$.

In what follows assume that $\sigma_{-1}\sigma_1,\sigma_2,\ldots,\sigma_{n-1}$ satisfy relations (DS1) to (DS4).

\begin{pro}\label{134}
The congruence $\overline{V}$ is generated by the following  relations:
\begin{enumerate}
\item$\sigma_i\sigma_j\th_i=\th_j\sigma_i\sigma_j$ with $i,j\in[n-1]\cup\{-1\}$ and $||i|-|j||=1$,\label{050}
\item$\sigma_i\sigma_j^{-1}\th_i=\th_j\sigma_i\sigma_j^{-1}$ with $i,j\in[n-1]\cup\{-1\}$ and $||i|-|j||=1$,\label{162}
\item$\sigma_i\th_j=\th_j\sigma_i$ with $i,j\in[n-1]\cup\{-1\}$ and $||i|-|j||\neq1$.\label{124}
\end{enumerate}
\end{pro}
\begin{proof}

Note that (\ref{050}) to (\ref{124}) hold in $T^DM$ and can be deduced from $V$. More specifically, (\ref{050}) and (\ref{162}) are obtained by taking $j=i+2$ in Corollary \ref{118}(\ref{121}) and $j=3$ in Corollary \ref{018}(\ref{165}). Relation (\ref{124}) is obtained by taking $j=i+1$ in Corollary \ref{118}(\ref{127}) and $(i,j)=(1,3)$ in Corollary \ref{118}(\ref{163}). So, it is enough to show that the relations in Corollary \ref{118} can be deduced from (\ref{050}) to (\ref{124}). Since \ref{117} holds for $\td$'s and we have $\oa_{1,j}^{-1}\td_{j-2}=\sigma_{-1}a_{2,j}$ and $a_{1,j}\td_{j-2}^{-1}=\sigma_{-1}^{-1}\oa_{2,j}^{-1}$, then, relations in Corollary \ref{118} are obtained from (\ref{050}) to (\ref{124}) in the same way as they were obtained in Proposition \ref{057}.
\end{proof}

Now, as we remove generators $\ve$'s and add $\th_{-1},\th_1,\ldots,\th_{n-1}$, we need to rewrite relations of $P_n^D$ in terms of these generators. Lemma \ref{120} implies that (PD1) can be deduced from the involution of $\th$'s. Further, (PD4) becomes trivial and, as in Lemma \ref{129}, (PD2) is deduced from the commutativity of $\th$'s.

\begin{lem}\label{132}
Relations in \normalfont{(PD3)} can be obtained from $\ve_{1,2}\ve_{1,3}=\ve_{1,2}\ve_{2,3}=\ve_{1,3}\ve_{2,3}$.
\end{lem}
\begin{proof}
Denote by (PD3)$(i,j,k)$ the relation $\ve_{i,j}\ve_{i,k}=\ve_{i,j}\ve_{j,k}=\ve_{i,k}\ve_{j,k}$ with $i,j,k\in[\pm n]$ satisfying $i<j<k$ and $|i|,|j|,|k|$ all different. As in Lemma \ref{154} we can show that each (PD3)$(i,j,k)$ can be obtained, via conjugation by $d_{|i|}b_{1,|j|}b_{2,|k|}$, from one of themselves with indexes in $\{(1,2,3),(-1,2,3),(-2,1,3),(-3,1,2)\}$. However\[\begin{array}{rcl}
\text{(PD3)}(-1,2,3)&=&\sigma_1\sigma_2\sigma_{-1}\text{(PD3)}(1,2,3)(\sigma_1\sigma_2\sigma_{-1})^{-1},\\
\text{(PD3)}(-2,1,3)&=&\sigma_{-1}\sigma_2\sigma_{-1}\text{(PD3)}(1,2,3)(\sigma_{-1}\sigma_2\sigma_{-1})^{-1},\\
\text{(PD3)}(-3,1,2)&=&\sigma_{-1}\text{(PD3)}(1,2,3)\sigma_{-1}^{-1}.\end{array}\]Therefore, in $T^DM$, (PD3) is a consequence of (PD3)$(1,2,3)$.
\end{proof}
\begin{lem}\label{164}
Relations in \normalfont{(PD5)} can be obtained from $\ve_{-1,2}\ve_{1,2}\ve_{-3,4}\ve_{3,4}=\ve_{-1,2}\ve_{1,2}\ve_{2,3}\ve_{3,4}$.
\end{lem}
\begin{proof}
Denote by (PD5)$(i,j,r,s)$ the relation $\ve_{-i,j}\ve_{i,j}\ve_{-r,s}\ve_{r,s}=\ve_{-a,b}\ve_{a,b}\ve_{b,c}\ve_{c,d}$ such that $\{i<j,r<s\}=\{a<b<c<d\}$. Clearly, these relations can be obtained by conjugating with $d_ab_{1,b}b_{2,c}b_{3,c}$ one of the following relations:\[\begin{array}{lll}
\text{(PD5)}(1,2,3,4),&\text{(PD5)}(1,3,2,4),&\text{(PD5)}(1,4,2,3),\\
\text{(PD5)}(2,3,1,4),&\text{(PD5)}(2,4,1,3),&\text{(PD5)}(3,4,1,2).\end{array}\]By commutativity, we can consider only the relations on the first line. Further, by conjugating, we obtain the following:\[\text{(PD5)}(1,2,3,4)=\sigma_2\text{(PD5)}(1,3,2,4)\sigma_2^{-1}=\sigma_2\sigma_3\text{(PD5)}(1,4,2,3)(\sigma_2\sigma_3)^{-1}.\]Therefore, we deduce that each relation in (PD5) is obtained from (PD5)$(1,2,3,4)$.
\end{proof}

As in Remark \ref{059}, because of Lemma \ref{132} and Lemma \ref{164}, in $T^DM$, relations (PD3) and (PD5) are equivalent to the following:\begin{equation}\label{166}\begin{array}{cl}
\th_i\th_j\sigma_i=\th_j\sigma_i\th_j=\sigma_i\th_j\th_i,&||i|-|j||=1,\\
\th_{-1}\th_1\th_3\td_2\th_3=\th_{-1}\th_1\th_2\th_3\td_2.
\end{array}\end{equation}

\begin{proof}[Proof of Theorem \ref{TMD}]
Relations (DB0) to (DB4) and (PD1) to (PD5) are obtained from $\BB_n^D$ and $P_n^D$ respectively. Relations (TD3), (TD4) and (TD8) are obtained from Proposition \ref{134} and so $\ve$'s and (PD4) are removed. We add (TD1) and (TD2) and remove (PD1) and (PD2) because they are deduced from the first ones. Finally, we add redundantly relations (TD5)  and (TD6), and so remove (PD3)  and (PD5) as in Equation \ref{166}.
\end{proof}

\section{Tied algebras}\label{S6}

In this section we give a mechanism to construct algebras from tied monoids, we call these algebras \emph{tied algebras}. This mechanism produce the Hecke algebras, bt--algebras of type A \cite{AiJu19}, as well to those bt--algebras of type B defined in \cite{Ma17,Fl19}. At the end of the paper we show out the tied algebra obtained from the monoid $T^B\BB_n^B$ and $T^D\BB_n^D$.

For a commutative ring $R$ with unity $1$, and a monoid $M$, we denote by $R[M]$ the monoid ring of $M$ over $R$. Note that $R$ can be regarded as a subring of $R[M]$; the unity of $R[M]$ is $1$ and $M$ can be considered as subset of $R[M]$. Given a presentation $\langle D\mid\Gamma\rangle$ of $M$, we shall regard $D$ as a subset of $M$ and $\Gamma$ as equalities of elements in $M$. Note that $R[M]$ can be presented by $D$ subject to the relations derived from $\Gamma$. For an action $f:M\to\End(N)$, we denote by $\Fix_f(m)$ the set of points of the monoid $N$ fixed by $m\in M$, that is $\Fix_f(m)=\{n\in N\mid f_m(n)=n\}$.
 
Suppose the following data:\begin{enumerate}
\item $(W,S)$ is a finite Coxeter group of type	$\Gamma$,
\item $M$ a monoid endowed with a finite presentation $\langle A\mid X\rangle$ and an epimorphism $\phi:\langle A\mid X\rangle\to W$ satisfying $\phi(A)=S$,
\item $P$ an idempotent commutative monoid endowed with a finite presentation $\langle B\mid Y\rangle$ and an action $\chi:W\to\End(\langle B\mid Y\rangle)$,
\item $\vert E_s\vert \leq 1$ for all $s\in S$, where $E_s:=\Fix_{\chi}(s) \cap B$.
\end{enumerate}
 
From (2) and (3) we have the action $\chi\circ\phi$ of $M$ on $\langle B\mid Y\rangle$. Observe that $\Fix_{\chi}(s)=\Fix_{\chi}(t)$ for $s,t\in S$ conjugate; further if $\pi(a)=s$ for some $a\in A$, then ${\rm Fix}_{\chi}(s)={\rm Fix}_{\chi\circ \phi}(a)$.

Let $T^{\Gamma}\!M$ be the tied monoid of $M$ respect to $P$ associated to $\phi$. Thus we may consider $M$, $W$ and $P$ as subset of the ring $R[T^{\Gamma}\!M]$.
 
From now on, for every $s$ in $S$, we set pairs of invertible elements $(\uu_s,\vv_s)\in R^2$ such that:\[(\uu_s,\vv_s)=(\uu_t,\vv_t)\text{ if $s$ and $t$ are conjugate in $W$.}\]

\begin{dfn}
The \emph{bt--algebra} associated to $R[T^{\Gamma}\!M]$, denoted by $\EE(T^{\Gamma}\!M)$, is defined as the quotient of $R[T^{\Gamma}\!M]$ factoring out the following relations:\[a^2=1+(\uu_s-1)e+(\vv_s-1)ea\qquad s\in S,a\in X_s,e\in E_s,\]where $X_s:= \{a\in X \mid \phi(a)=s\}$ and $e$ is taken as $1$ if $E_s=\emptyset$.
\end{dfn}
 
Note that $e$ commutes with $a$ for $e\in E_s$, $a\in X_s$, hence the above quadratic relation is equivalent to\begin{equation}\label{055}a^{-1}=a+(1-\vv_s)\uu_s^{-1}e+(\uu_s^{-1}-1)ea\quad\text{(cf. \cite[(20)]{AiJu19})}.\end{equation}

We will see below that $\EE(T^{\Gamma}\!M)$ capture the Hecke algebras and the bt--algebra with two parameters defined in \cite{AiJu19}. 

The Hecke algebra associated to a Coxeter group $W$ corresponds to take in the data of $\EE(T^{\Gamma}\!M)$: $M$ as the Artin group attached to $W$, and $P$ as the trivial monoid; all of them with the usual presentations and the natural actions.

The bt--algebra with two parameters $\EE_n(\uu,\vv)$, correspond to $\EE(T\!\BB_n)$. More precisely, it is obtained by taking in the data: (1) $W=\mathfrak{S}_n$ and $S$ formed by the transpositions $s_i=(i,i+1),\,i\in [n-1]$, (2) $M=\BB_n$ and $\langle A\mid X\rangle$ its usual presentation, i.e. $A$ is formed by the $\sigma_i$'s and $X$ by braid relations (B0) to (B2), and $\phi:\sigma_i\mapsto s_i$, finally (3) $P=P_n$ with the presentation $\langle B\mid Y\rangle$ of Theorem \ref{001} and $\chi$ the natural action of $W$ on $P_n$. Notice that $E_s=\{\eta_i\}$. Set now $(\uu,\vv)=(\uu_s,\vv_s)$ for all $s\in S$. Then, with the previous ingredients and Proposition \ref{014}, $\EE(TM)$ results to be the algebra presented by the generators $a_i$'s and $\eta_i$'s satisfying the relations of  Proposition \ref{014} (changing $\sigma_i$ by $a_i$), together with the relations:\[a_i^2=1+(\uu-1)\eta_i+(\vv-1)\eta_ia_i,\qquad i\in[n-1].\]By using \ref{055} and relations (T3) and (T5), we deduce that (T6) can be omitted in the defining presentation of $\EE(TM)$. Therefore, $\EE(TM)$ coincides with the bt--algebra defined in \cite[Definition 2]{AiJu19}.

Now, with the same notation of Subsection 4.2, we are going to explicit $\EE(T^B\!M)$ for $M=\BB_n^B$. First, set $W=\SS_n^B$ and $S=\{t_0,t_1,\ldots,t_{n-1}\}$. Second, $M$ is considered with the presentation $\langle A\mid X\rangle$ associated to $S$ and $\phi$ the natural epimorphism mapping $A$ onto $S$. Third, $P$ is taken as the monoid $P_n^B$ with the presentation $\langle B\mid Y\rangle$ yielded in Theorem \ref{087}     and set $\chi$ the action of $\SS_n^B$ on $P_n^B$ obtained from the natural action of $\SS_{\pm n}$ on $P_{\pm n}$. Now, we have $E_{t_0}=\{\ve_{-1,1}\mid i\in [n]\}$ and $E_{t_k}=\{\ve_{k,k+1}\}$ for $k\in [n-1]$. Set $(\uu,\vv),(\pp,\qq)\in R^2$. Therefore, the algebra $\EE(T^B\BB_n^B)$ is defined by braids generators $a_0,a_1,\ldots,a_{n-1}$ and the tied generators $\th_0,\th_1,\ldots,\th_{n-1}$ subject to the relations of Theorem \ref{031} and the relations:\[a_0^2=1+(\pp-1)\th_0+(\qq-1)\th_0a_0,\quad a_i^2=1+(\uu-1)\th_i+(\vv-1)\th_ia_i,\quad i\in[n-1].\]

Notice that due to \ref{055} and (TB6), the relations (TB8) and (TB9) are redundant.

Finally, observe that the tied algebra $\EE(T^D\BB_n^D)$ attached to $T^D\BB_n^D$ turns out to be presented by generators $a_{-1},a_1,a_2\ldots,a_{n-1}$, $\th_{-1},\th_1, \th_2\ldots,\th_{n-1}$ subject to the relations (DB1) to (DB2) among the $a_i$'s, the relations (TD1) to (TD7) of Theorem \ref{TMD}, together with the relations\[a_i^2=1+(\uu-1)\th_i+(\vv-1)\th_ia_i,\quad i\in\{-1\}\cup[n-1].\]Notice that due to \ref{055} and (TD5) the relation (TD8) is redundant. Further, we have the algebra inclusion\[\EE(T^{\rm D}\BB_n^D)\subset\EE(T^{\rm B}\BB_n^B)',\]where $\EE(T^B\BB_n^B)'$ means that in $\EE(T^B\BB_n^B)$ is taken $\pp=\qq=1$. 

\subsubsection*{Acknowledgements}

We would like to thank Prof. F. Aicardi for her important observations and many useful comments to improve our manuscript.  We also thank Prof. J. East for his kind comments about the first version of the manuscript.

The first author is part of the research group GEMA Res.180/2019 VRIP--UA and was supported, in part, by the grant Fondo Apoyo a la Investigaci\'on DIUA179-2020. The second author was supported, in part, by the grant FONDECYT Regular Nro.1180036.

\bibliographystyle{plain}
\bibliography{../../../Projects/bibtex}

\begin{thebibliography}{10}

\bibitem{AiJu00}
F.~Aicardi and J.~Juyumaya.
\newblock An algebra involving braids and ties.
\newblock ICTP Preprint IC/2000/179, 2000.

\bibitem{AiJu16}
F.~Aicardi and J.~Juyumaya.
\newblock Tied links.
\newblock {\em J Knot Theor Ramif}, 25(9):1641001, 2016.

\bibitem{AiJu18}
F.~Aicardi and J.~Juyumaya.
\newblock Kauffman type invariants for tied links.
\newblock {\em Math Z}, 289(1--2):567--591, 6 2018.

\bibitem{PreAiJu18}
F.~Aicardi and J.~Juyumaya.
\newblock Tied links and invariants for singular links.
\newblock Preprint, 2018.

\bibitem{AiJu19}
F.~Aicardi and J.~Juyumaya.
\newblock Two parameters bt-algebra and invariants for links and tied links.
\newblock {\em Arnold Math J}, 2020.
\newblock to appear.

\bibitem{ArPa19}
D.~Arcis and L.~Paris.
\newblock Ordering {G}arside groups.
\newblock {\em Int J Algebr Comput}, 29(5):861--883, 3 2019.

\bibitem{Ar25}
E.~Artin.
\newblock Theorie der {Z}\"{o}pfe.
\newblock {\em Abh Math Sem Hamburg}, 4(1):47--72, 12 1925.

\bibitem{Ar47}
E.~Artin.
\newblock Theory of {B}raids.
\newblock {\em Ann Math}, 48(1):101--126, 1 1947.

\bibitem{Ba92}
J.~Baez.
\newblock Link invariants of finite type and perturbation theory.
\newblock {\em Lett Math Phys}, 1(26):43--51, 9 1992.

\bibitem{Bi93}
J.~Birman.
\newblock New points of view in knot theory.
\newblock {\em B Am Math Soc}, 11(2):253--287, 4 1993.

\bibitem{BjBr05}
A.~Bjorner and F.~Brenti.
\newblock {\em Combinatorics of {C}oxeter groups}.
\newblock Graduate Texts in Mathematics 231. Springer Science+Business Media,
  Inc., New York, 1 edition, 2005.

\bibitem{CPM16}
C.~Caprau, A.~de~la Pena, and S.~McGahan.
\newblock Virtual singular braids and links.
\newblock {\em Manuscripta Math}, 151(1--2):147--175, 9 2016.

\bibitem{CJKL18}
M.~Chlouveraki, J.~Juyumaya, K.~Karvounis, and S.~Lambropoulou.
\newblock Identifying the invariants for classical knots and links from the
  {Y}okonuma--{H}ecke algebras.
\newblock {\em Int Math Res Notices}, 2020(1):214--286, 1 2020.

\bibitem{Fi03}
D.~Fitz{G}erald.
\newblock A presentation for the monoid of uniform block permutations.
\newblock {\em B Aust Math Soc}, 68(2):317--324, 10 2003.

\bibitem{Fl19}
M.~Flores.
\newblock A braids and ties algebra of type {$B$}.
\newblock {\em J Pure Appl Algebra}, 224(1):1--32, 1 2020.

\bibitem{JaPo17}
N.~Jacon and L.~Poulan d'Andecy.
\newblock Clifford theory for {Y}okonuma--{H}ecke algebras and deformation of
  complex reflection groups.
\newblock {\em J London Math Soc}, 96(3):501--523, 9 2017.

\bibitem{Kd04}
S.~Kamada.
\newblock Invariants of virtual braids and a remark on left stabilisations and
  virtual exchange moves.
\newblock {\em Kobe J Math}, 21(1-2):33--49, 2004.

\bibitem{Ka99}
L.~Kauffman.
\newblock Virtual {K}not {T}heory.
\newblock {\em Eur J Combin}, 20(7):663--691, 10 1999.

\bibitem{KaLa04}
L.~Kauffman and S.~Lambropoulou.
\newblock Virtual braids.
\newblock {\em Fund {M}ath}, 184(1):159--186, 2004.

\bibitem{Ll79}
G.~Lallement.
\newblock {\em Semigroup and Combinatorial Applications}.
\newblock Pure and Applied Mathematics. John Wiley \& Sons, Inc., New York,
  1979.

\bibitem{La98}
T.~Lavers.
\newblock Presentations of {General} {Products} of {Monoids}.
\newblock {\em J Algebra}, 204(2):733--741, 6 1998.

\bibitem{Ma17}
I.~Marin.
\newblock Artin groups and {Y}okonuma--{H}ecke algebras.
\newblock {\em Int Math Res Notices}, 2018(13):4022--4062, 7 2018.

\bibitem{Ma18}
I.~Marin.
\newblock Lattice extensions of {H}ecke algebras.
\newblock {\em Journal of Algebra}, 503:104--120, 6 2018.

\bibitem{Or00}
F.~Oravecz.
\newblock Symmetric partitions and pairings.
\newblock {\em Colloq Math}, 86(1):93--101, 2000.

\bibitem{Re97}
V.~Reiner.
\newblock Non-crossing partitions for classical reflection groups.
\newblock {\em Discrete Math}, (1--3):195--222, 12 1997.

\bibitem{RoGa99}
J.~Rosales and P.~Garc\'ia.
\newblock {\em Finitely {G}enerated {C}ommutative {M}onoids}.
\newblock Nova Science Publishers Inc., New York, 1 edition, 1999.

\bibitem{Ru95}
N.~Ru\v{s}kuc.
\newblock {\em Semigroup Presentations}.
\newblock PhD thesis, University of St. Andrews, 4 1995.

\bibitem{RyH11}
S.~Ryom-Hansen.
\newblock On the representation theory of an algebra of braids and ties.
\newblock {\em J Algebr Comb}, 33(1):57--79, 2 2011.

\bibitem{Sm88}
L.~Smolin.
\newblock Knot theory, loop space and the diffeomorphism group.
\newblock In {\em New Perspectives in Canonical Gravity}, volume~5 of {\em
  Monogr Textbooks Phys Sci}, chapter~6, pages 245--266. Bibliopolis, Naples, 2
  1988.

\end{thebibliography}



\end{document}